\documentclass[a4paper]{article}
\usepackage[latin1]{inputenc}
\usepackage{amsmath,amssymb,amsthm}
\usepackage[headheight=12pt]{geometry}
\usepackage{url}
\usepackage{hyperref} 
\usepackage{graphicx}
\long\def\beginpgfgraphicnamed#1#2\endpgfgraphicnamed{\includegraphics{#1}}%
\usepackage[T1]{fontenc}
\title{Resonances for a diffusion with small noise}
\author{Markus Klein%
\thanks{Universität Potsdam, Institut für Mathematik,
Professur Mathematische Physik Semiklassik \& Asympthotik,
14415 Potsdam}, 
Pierre-André Zitt%
\thanks{Équipe d'accueil Modal'X, b\^{a}t. G,
Université Paris \textsc{X}, 92000 Nanterre}%
}
\newcounter{stepcnt}
\newcommand{\step}[1]{%
  \stepcounter{stepcnt}%
  \par\noindent\textbf{Step \thestepcnt: #1.} %
  }
\newcommand*{\abs}[1]{\left|#1\right|} 
\newcommand*{\field}[1]{\mathbb{#1}}
\newcommand*{\R}{\field{R}}
\let\xR\R

\let\xC\C
\newcommand*{\xN}{\mathbb{N}}
\newcommand{\xRe}{\mathfrak{R}}
\newcommand{\xIm}{\mathfrak{I}}
\newcommand*{\bigO}{O}
\newcommand*{\logSim}{\stackrel{ln}{\sim}}
\newcommand{\grad}{\nabla}
\newcommand{\on}[1]{_{|#1}}
\newcommand*{\norm}[1]{\left\|#1\right\|}

\newcommand*{\lTwoNrm}[1]{\left\|#1\right\|_2}
\newcommand*{\compl}[1]{#1^{c}}
\newcommand*{\floor}[1]{\lfloor#1\rfloor}
\newcommand*{\tens}{\otimes}
\newcommand*{\ind}[1]{\mathbf{1}_{#1}}
\newcommand*{\aleq}{\lesssim}
\DeclareMathOperator{\supp}{Supp}
\DeclareMathOperator{\rank}{Rank}

\newcommand*{\fou}[1]{\hat{#1}}
\newcommand*{\cf}{cf.}
\newcommand*{\ie}{\emph{i.e.}}
\newcommand*{\eg}{\emph{e.g.}}
\newtheorem{thrm}{Theorem}
\newtheorem{prpstn}[thrm]{Proposition}
\newtheorem{dfntn}[thrm]{Definition}
\newtheorem{rmrk}[thrm]{Remark}
\newtheorem{hpthss}{Hypothesis}
\newtheorem{lmm}[thrm]{Lemma}
\newcommand*{\eps}{\varepsilon}         
\newcommand*{\rZero}{r_0(\eps)}         
\newcommand{\op}{H}
\newcommand*{\fullOp}{H_\eps}           
\newcommand*{\intOp}{H^{\rZero}_i}      
\newcommand*{\extOp}{H^{\rZero}_e}      
\newcommand*{\extRotOp}{\extOp(\theta)} 
\newcommand*{\rotOp}{H(\theta)}         
\newcommand*{\fixedOp}{H^R_i}           
\newcommand*{\extDResolv}{R^D_e}        
\newcommand*{\LB}{\Delta_{LB}}          
\newcommand*{\PDrchlt}{P^D}             
\newcommand*{\QDrchlt}{Q^D}             
\DeclareMathOperator{\tr}{trace}        
\newcommand*{\ball}[1]{\mathcal{B}(#1)} 
\newcommand*{\fixedBall}{\ball{R}}      
\newcommand*{\largeBall}{\ball{\rZero}} 

\newcommand{\chiTldZero}{\tilde{\chi}_0}
\newcommand*{\PDO}{$\Psi D O$}
\newcommand*{\pdo}{Op}

\newcommand*{\agmon}{d_{Ag}}
\newcommand{\minima}{\mathcal{M}}
\newcommand{\deltaTilde}{\tilde{\delta}}
\newcommand{\tldV}{\tilde{V}_{\omega,\eps}}
\newcommand{\tldLambda}{\tilde{\lambda}}
\newcommand{\tldMu}{\tilde{\mu}}

\newcommand{\devPart}[2]{\frac{\partial #1}{\partial #2}}
\newcommand{\devSec}[2]{\frac{\partial^2 #1}{\partial #2^2}}
 \newcommand{\dXAlpha}{\partial_x^\alpha}
 \newcommand{\dXAlphaDXiBeta}{\partial_x^\alpha \partial_\xi^\beta}

\newcommand{\safetyCst}{c_{S}}
\newcommand{\chiSmooth}{\chi_{sm}}

\begin{document}
\maketitle
\begin{abstract}
    We study resonances for the generator of a diffusion with small noise in $\xR^d$~: $L_\eps =  - \eps\Delta + \nabla F\cdot\nabla$, when the potential $F$ grows slowly at infinity (typically as a square root of the norm). The case when $F$ grows fast is well known, and under suitable conditions one can show that there exists a family of exponentially small eigenvalues, related to the wells of $F$. We show that, for an $F$ with a slow growth, the spectrum is $\xR_+$, but we can find a family of resonances whose real parts behave as the eigenvalues of the ``quick growth'' case, and whose imaginary parts are small.
\end{abstract}
\textbf{Acknowledgments.} We gratefully acknowledge the financial support of the DFH--UFA (franco--german university), through the CDFA/dfD 01--06, ``Applications of stochastic processes''. 
\section{Introduction}
The aim of this paper is to understand, from a spectral point of view, a probabilistic result obtained by one of us (\cite{Zit07annealing}). This result is the convergence of an ``annealing diffusion'', a process defined by the stochastic differential equation
\[
dX_t = \sqrt{\sigma(t)} dB_t - \grad F(X_t)dt,
\]
where $F:\xR^d \to \xR$ is a function to be minimized, and the ``temperature'' $\sigma(t)$ is a deterministic function going to zero (the ``convergence'' means that $X$ finds the global minima of $F$). The convergence was already known for potentials with a quick growth (the typical case being $F = \abs{x}^a, a>1$ at infinity); we generalized it to the ``slow growth'' case  (when $F$ behaves like $\abs{x}^a$, with $a<1$). In this case, the classical approach using strong functional inequalities (log-Sobolev, Poincaré) breaks down, and we had to resort to the so-called ``weak Poincaré inequality''.

What we are interested in here is a spectral traduction of this convergence result. 
In the ``quick growth'' case, it is known that  the spectral gap of an ``instantaneous equilibrium measure'' (equilibrium for a process at fixed temperature $\sigma$) is related to the depth $d$ of a certain well of $F$, so that:
\begin{equation}
  \label{eq=asymptoticSpectralGap}
  \text{Spectral gap at temperature }\sigma \approx \exp \left( - \frac{d}{\sigma} \right). 
\end{equation}
This can be used to find the optimal choice of $\sigma$ (namely $\sigma(t) = d/\log(t)$).

In the ``slow growth'' case of \cite{Zit07annealing}, the instantaneous measure \emph{do not} have a spectral gap. However, in a sense, they behave as if they did: the same choice of the freezing schedule, $\sigma(t) = \frac{d}{\log(t)}$, still guarantees convergence.

\bigskip

To be more precise, let us recall here another set of results, focused on the behaviour of the lower spectrum of the operators
\[  - \sigma \Delta + \nabla F\cdot \nabla,\]
when $\sigma\to 0$. In other words, we consider the generators of the original SDE (with a conventional minus sign), forgetting non-stationarity ($\sigma$ is fixed), but looking at the asymptotic $\sigma\to0$. Once more, when $F$ grows fast at infinity, much is known: using probabilistic (\cite{BEGK04,BGK05}) or, to refine the results, analytic techniques (\cite{HKN04,HN06}), a very precise analysis of the lower spectrum has already been done (we will recall and use one of these results in theorem \ref{thrm=interiorSpectrumFixedBall}). In particular, these results contain and precise the asymptotic \eqref{eq=asymptoticSpectralGap}.

In this case, the convergence result for the annealing process can be ``seen'' on the lower spectrum of the operators: the optimal freezing schedule is dictated by the asymptotic behaviour of the lower spectrum, via the constant $d$ in \eqref{eq=asymptoticSpectralGap}.

Let us remark here that the explicit terms in the asymptotic developments all depend on ``local'' properties of $F$, \ie{}  its structure on a compact set, and not on the details of its growth at infinity.

\bigskip

In the ``slow growth'' case,  the spectra are always $\xR^+$: the simulated annealing process still converges, but its optimal freezing schedule seems to be disconnected from the spectral properties of the generators. We prove that, under certain circumstances, it is not, by exhibiting other spectral quantities with the correct order of magnitude: in other words, we will try to understand what becomes of the small eigenvalues when we ``change'' the growth rate of $F$ at infinity.

We will see that the former eigenvalues give rise to \emph{resonances}.

\bigskip

Resonances are, in some sense, what remains of eigenvalues when eigenvalues
disappear. We refer to \cite{Zwo99} for a very nice introduction to resonances
(with examples from PDEs and quantum mechanics). Let us give an idea of what resonances are (the precise definition we will use comes from a different point of view, see remark \ref{rmrk=aboutTheDefOfResonances}). They may be seen as singular values, not of the resolvent map
itself (like usual eigenvalues), but of an analytic continuation of the
resolvent map on a certain dense subspace.

To be more precise, for an operator
$L$, the usual resolvent is $R(\lambda) = (L - \lambda)^{-1}$. Suppose that
$\sigma(L) = \xR_+$, and consider the map:
\[
  \lambda \mapsto (\phi, R_\lambda \phi).
\]
For a given $\phi$, this function may have a meromorphic continuation across
the real axis. If the continuations, for all $\phi$ in a dense subset, have a
common pole at some complex number $\mu$ ($\mu$ lies on the lower half-plane,
see figure \ref{fig=resonances}), this pole is called a \emph{resonance}. 

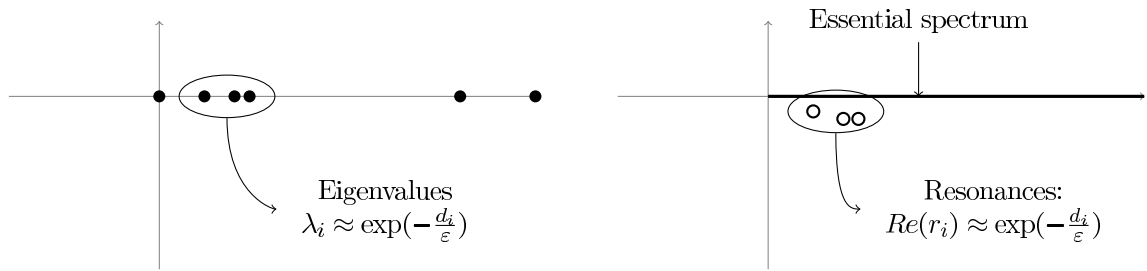
\begin{figure}
    \newcommand{\tb}[1]{\begin{tabular}{c}#1\end{tabular}}
  \beginpgfgraphicnamed{resonances_1}%
\begin{tikzpicture}
    \draw[help lines,->] (-2,0) -- (5,0);
    \draw[help lines,->] (0,-2.3) -- (0,1);
    \fill[thick] (0,0) circle (0.08cm);
    \fill[thick] (0.6,0) circle (0.08cm);
    \fill[thick] (1,0) circle (0.08cm);
    \fill[thick] (1.2,0) circle (0.08cm);
    \fill[thick] (4,0) circle (0.08cm);
    \fill[thick] (5,0) circle (0.08cm);
    \node[draw,ellipse,inner xsep=0.45cm,inner ysep=0.2cm] (vp) at (0.9,0) {};
    \node (label) at (3,-1.5) {\tb{Eigenvalues\\ $\lambda_i \approx \exp (-\frac{d_i}{\eps})$}};
    \draw [->] (vp) ..controls +(down:1.3cm) and +(left:1.5cm) .. (label);
\end{tikzpicture}%
\endpgfgraphicnamed
\hfill
\beginpgfgraphicnamed{resonances_2}%
\begin{tikzpicture}
    \draw[help lines,->] (-2,0) -- (5,0);
    \draw[help lines,->] (0,-2.3) -- (0,1);
    \draw[very thick] (0,0)--(5,0);
    \node (ess sp) at (2,1) {Essential spectrum};
    \draw[->] (ess sp) -- (2,0);
    \draw[thick] (0.6,-0.2) circle (0.08cm);
    \draw[thick] (1,-0.3) circle (0.08cm);
    \draw[thick] (1.2,-0.3) circle (0.08cm);
    \node[draw,ellipse,inner xsep=0.45cm,inner ysep=0.2cm] (vp) at (0.9,-0.2) {};
    \node (label) at (3,-1.5) {\tb{Resonances: \\$Re(r_i) \approx \exp (-\frac{d_i}{\eps})$}};
    \draw [->] (vp) ..controls +(down:1.3cm) and +(left:1.9cm) .. (label);
\end{tikzpicture}%
\endpgfgraphicnamed
\label{fig=resonances}
\caption{Resonances}
\end{figure}

To prove the existence of such quantities, a possible approach is to deform
drastically the operator and try to move the essential spectrum ``out of the
way''. Resonances of the original operator then appear as (complex) eigenvalues
of the (non self-adjoint) deformed operator. A basic example of this technique,
called ``complex scaling'', can be found in Reed and Simon \cite{RS78_IV} (sections \textsc{XII}.6 and \textsc{XIII}.10).
However, what bothers us in our case is really the part coming from infinity,
and we would like to keep the operator intact on the region where the minima
are. Therefore we will have to resort to the more refined exterior complex
scaling (see below for more remarks on this).

\begin{rmrk}
    \label{rmrk=aboutTheDefOfResonances}
    In the sequel, we will define resonances as the eigenvalues of the distorted operator (\cf{} section \ref{subsec=exteriorScaling}).
\end{rmrk}

To be able to adapt known results on resonances more easily, we perform a unitary transform of our operators that turn them into the following Schrödinger operators:
\begin{equation}
  \label{eq=defFullOperator}
  \fullOp = - \eps^2 \Delta + V_\eps
\end{equation}
where
\[
V_\eps(x)  =  \frac{1}{2}\abs{\nabla F}^2 - \frac{\eps}{2}\Delta F ,
\]
and $F$ is the original probabilistic potential. This correspondance is well
known (originally, it was noted and used to study Schrödinger operators,
\cf{} for example \cite{Car79}; later the reverse way was also used, for
example in \cite{Cat05} to prove criteria for the spectral gap).

We also define 
\begin{equation}
  \label{eq=VandF}
  V = \frac{1}{2} \abs{\nabla F}^2 .
\end{equation}

The paper is divided in the following way. First we state our hypotheses and the main result (section \ref{sec=mainResult}). Section \ref{sec=extScaling} describes exterior scaling and how it is used to prove the existence of resonances: we will define here several auxiliary operators, obtained by putting a Dirichlet boundary condition on a particular sphere, and modifying further the ``outside'' part.

In section \ref{sec=Agmon}  we  prove estimates on the decay of eigenfunctions of certain operators, in the spirit of Agmon.
 We describe in section \ref{sec=lowerSpectrum} the lower spectrum of the interior operator. 
 We need to  show that the exterior part of the operator does not create resonances near the eigenvalues (which come  from the interior part). This is one of the main difficulties;  it is done in section \ref{sec=exteriorResolvent}, using symbolic calculus for pseudo-differential operators.

Finally, all these results are put together in section \ref{sec=spectralStability}, where we establish a ``spectral stability'' between the original operator and the modified one, thereby proving the existence of resonances.

\bigskip
\textbf{Notation.}
Almost every quantity we will consider will depend on the small parameter $\eps$ . For two such quantities $a$ and $b$, we write $a\aleq b$ if there exists a constant $C$ such that $a\leq Cb$. This constant may depend on the dimension $d$, and  on the potential $F$, but not on $\eps$. We will also write $a\logSim b$ if $\log(a)\sim\log(b)$.

\section{Main result}
\label{sec=mainResult}
  We need two kind of hypotheses on the ``probabilistic potential'' $F$
  (and on the ``Schrödinger potential'' $V_\eps$): some describe the well structure inside a compact set, the others deal with behaviour at infinity.

  The first ones are in some sense ``non degeneracy'' assumptions, that were used in the ``quick growth'' case (\cite{BEGK04,BGK05,HKN04}, to which we refer for more details). To state them, we need a definition: for a point $x$ and a set $A$, let $\mathcal{C}(x,A)$ be:
  \[
  \inf_{\gamma\in\Gamma(x,A)}\{ \sup_{t\in[0,1]} F(\gamma(t))\}  - F(x),
  \]
  where the $\Gamma$ is the set of continuous paths joining $x$ to $A$. 
  The quantity $\mathcal{C}(x,A)$ is the ``cost'' one has to pay to go from $x$ to $A$; in other words, this is the height of the energy barrier between $x$ and $A$.

  In terms of this cost, the assumptions may be formulated as follows:
  \begin{itemize}
      \item $F$ has a finite number of local minima, $x_0,\ldots x_N$.
      \item $x_0$ is the (unique) global minimum.
      \item there exist critical depths $d_0=\infty > d_1 >\cdots d_N$ such that:
          \[ \mathcal{C}\left(x_{i+1},\{x_0,\ldots x_i\} \right)= d_i.\]
  \end{itemize}
  \begin{rmrk}
      It is natural to set $d_0 = \infty$: it corresponds to the
      cost of going from the global minimum to infinity (where $F \to \infty$).
      Furthermore, we will associate to each potential well  an eigenvalue of
      order $\exp(-d_i/\eps)$: the first eigenvalue $0$ therefore corresponds
      to the global minimum (the infinitely deep well). Finally, we will have
      to consider a (simple) case with boundary later on: we will then
      introduce a $d'_0$, the cost of going from the global minimum to the
      boundary, which will describe the lowest lying eigenvalue.
  \end{rmrk}

We now state the assumptions on the behaviour at infinity of $V$. Let us note beforehand that they seem much more stringent than the ``local'' ones. However, in the light of the original probabilistic result, it is already interesting to know what happens in the ``reference case'' where $F(x) = \abs{x}^a$ at infinity,  with $0<a<1$. 

The ``exterior scaling'' method demands that $V_\eps$ has an analytic
continuation somewhere near the ``real axis'' $\xR^ n$. We assume it in a
small conic region. To define it, let $\xIm(z),\xRe(z)$ denote the real and
imaginary parts of $z\in\xC^d$ (if $z = (x_1 + iy_1, \ldots x_d + iy_d)$,
$\xRe(z)$ is the vector $(x_1,\ldots x_d)$, and $\abs{\xRe(z)}$ is its
euclidean norm in $\xR^d$).
\begin{hpthss}
  \label{hpthss=analyticity}
  There exists an angle (say $3\beta_0$) such that 
$F$, as a function on the exterior of a fixed ball,   
has an analytic continuation to the following subset of $\xC^d$: 
\begin{equation}
    \label{eq=defAnalyticityCone}
    \mathcal{S} = \left\{ z \bigg| \abs{z}\geq R_0,  \frac{ \abs{\xIm(z)}}{\abs{\xRe(z)}} \leq \tan(2\beta_0) \right\}.
\end{equation}
\end{hpthss}
\begin{rmrk}
  In some references, it  is the map $r \mapsto V(r,\omega)$  that is supposed
  to be analytic (where $V$ is seen as a multiplication operator on $L^2$).
  For simplicity, we assume analyticity directly for $F$, and therefore for
  $V_\eps$. This will also give us estimates on the derivatives of $V$.
\end{rmrk}

\begin{hpthss}
  \label{hpthss=decay}
  $V_\eps$ has a power-law decay at infinity: there exists $ \gamma \in (0,2), c_V, C_V$, independent of $\eps$, such that
  \[ c_V \abs{x}^{- \gamma} \leq V_\eps(x) \leq C_V \abs{x}^ {-\gamma}, \]
  outside some fixed ball. Its analytic continuation is similarly bounded:
  \[ \abs{V'(z)}\leq C_V \abs{z}^{-\gamma}.\]
  Moreover, outside this ball, and for any angle $\omega$, 
  \[ \abs{V_\eps'(r)} \geq c_V V_\eps(r).\]
\end{hpthss}

\begin{rmrk}
  This hypothesis is very strong. However,  such bounds do seem necessary
  if we are to accurately estimate the deformation $V(r_\theta,\omega)$
  (\cf{} in particular proposition \ref{prpstn=TaylorBoundsOnRTheta}).
  The lower bound on $V'$ is reminiscent of so-called ``non-trapping
  conditions'' (\cf{} remark \ref{rmrk=aboutNonTrapping} below). The restriction $\gamma<2$ is more natural than it seems: we will see later that, if $\gamma >2$, the Agmon distance between the wells and infinity becomes finite, so the approach should break down. In any case, this hypothesis covers the reference case $F = \abs{x}^a$.
\end{rmrk}

We now come to the statement of the main result. It uses the distorted operator $\rotOp$, which will be formally defined in the next section (eq. \eqref{eq=expression_of_the_exterior_scaling}.
  \begin{thrm}
    \label{thm=mainResult}
    There exist $\theta = i\beta$,  some  functions $\rZero, S(\eps)$ and, for each index $i$, two functions $\lambda_i(\eps)$ and  
    $\mu_i(\eps)\in\xC$, such that:
    \begin{itemize}
        \item $\lambda_i$ is a Dirichlet eigenvalue for $\eps^2\Delta + V_\eps$ in the ball of radius $\rZero$, 
        \item $\mu_i$ is an eigenvalue of the distorted operator $\rotOp$ (\ie{} a resonance of $H_\eps$),
        \item these quantities satisfy: 
            \begin{align*}
    \abs{\xRe(\mu_i) - \lambda_i(\eps)} &\leq \exp(-S(\eps)/\eps) \\
    \abs{\xIm (\mu_i) } &\leq \exp( -S(\eps)/\eps),\\
    \lambda_i(\eps) &\logSim \exp( - d_i / \eps),
    \end{align*}
\item $S(\eps)$ goes to infinity.
    \end{itemize}
  \end{thrm}
  Therefore, we have identified spectral quantities (resonances) $\mu_i$ with the right asymptotic behaviour: their real part is of order $\exp(-d_i/\eps)$, and their imaginary part is much smaller (since $S(\eps)\to \infty$).
  
  Before we go on to the proof, let us mention that we did not address the problem of a probabilistic interpretation of the resonances: we only prove that their asymptotic behaviour is related to the depths of the wells of $F$, which are in turn related to mean exit times from these wells. 


\section{Exterior scaling} 
\label{sec=extScaling}
The exterior dilation (or scaling) is a technical device that allows one to see resonances of an operator as an eigenvalue of some (non self-adjoint) dilated operator. Intuitively, the operator is unchanged inside a large region, but is modified outside it by a change of scale (hence the name).

Let us first use hypothesis \ref{hpthss=decay} to define the region  we will use.
\begin{dfntn}
  \label{dfntn=rEps}
  Let $\rZero = \left(\frac{c_V}{\eps}\right) ^{\frac{1}{\gamma}} $.

  Then $\rZero\to \infty$, and on the sphere  $\partial\largeBall $, $V_\eps(x)\geq \eps$.
\end{dfntn}
This choice of the ball is guided by two constraints:
\begin{itemize}
  \item It must be far enough from the critical values of $F$ (we will see later that the Agmon distance between this ball and the critical points of $F$ should go to infinity),
  \item $V$ on the boundary should be large w.r.t the order of the lower eigenvalues (here $V\approx \eps$, whereas the eigenvalues are exponentially small).
\end{itemize}

The proper way to define exterior scaling is to use polar coordinates.
\subsection{The operators in polar coordinates}
We express the exterior dilation transformation in polar coordinates. To simplify notations, we drop here the dependence on $\eps$ and write $r_0$.
\newcommand{\extSpace}{\compl{\ball{r_0}}}
\newcommand{\extRSpace}{[r_0,\infty)}
\newcommand{\sphere}{\mathcal{S}_{n-1}}
\newcommand{\cartesianLTwo}{L^2(\extSpace)}
\newcommand{\polarLTwo}{L^2(\extRSpace\times\sphere)}

We introduce the change of coordinates:
\[
\begin{aligned}
  \extSpace &\to \extRSpace\times\sphere \\
  x &\mapsto (r(x),\omega(x))
\end{aligned}.
\]
To $f$ we associate $\tilde{f}:(r,\omega) \mapsto f(r\omega)$, and for any $x$, $\tilde{f}_x:\omega \mapsto \tilde{f}(r(x),\omega)$.

The Laplacian decomposes as the sum of a radial operator and a spherical operator, as follows:
\begin{equation}
    \label{eq=known_decomposition}
\Delta f(x) = \frac{1}{r^{n-1}} \devPart{\tilde{f} }{ r} (r^{n-1} \devPart{\tilde{f}}{ r} \tilde{f} )(r(x),\omega(x)) + \LB \tilde{f}_x ( \omega(x)).
\end{equation}
where $\LB$ is the Laplace Beltrami operator on the sphere $S^{n-1}$.

Since we would like the change of coordinates to be unitary in $L^2$, we use a slightly different choice:
\[
  \begin{aligned}
  \mathcal{O}: \cartesianLTwo &\to \polarLTwo \\
  f &\mapsto Of : (r,\omega) \mapsto r^{(n-1)/2} f(r\omega).
  \end{aligned}
\]
In turn, this defines (by conjugation) an equivalence between operators in the two $L^2$ spaces. We also note that the second space can be identified with the tensor product $L^2(\extRSpace)\tens L^2(\sphere)$.

We look for an expression of $\Delta$. It is easy to see that $\devPart{}{r}$ (in polar coordinates)
corresponds to $Df(x)  = (n-1)/(2r(x)) + \omega(x)\cdot \nabla f(x)$ (in other words, $   O D O^{-1} = \devPart{}{r}$).  An easy computation yields 
\begin{align*}
   D^2 f &= \frac{(n-1)(n-3)}{4r^2} f  + \frac{n-1}{r} \omega\cdot \nabla f + \omega\cdot \nabla(\omega\cdot \nabla f) \\
   &=       \frac{(n-1)(n-3)}{4r^2} f  + \frac{n-1}{r} \devPart{\tilde{f}}{r} + \devSec{\tilde{f}}{r} \\
   &=       \frac{(n-1)(n-3)}{4r^2} f  + \Delta f - \frac{1}{r^2} \LB \tilde{f}.
 \end{align*}

Since $ Of (r,\omega) = r^{(n-1)/2} \tilde{f}(r,\omega)$, 
\[ \LB \tilde{f}_x(\omega) = r(x)^{-(n-1)/2} \LB (Of) (r(x),\omega).\]
The decomposition \eqref{eq=known_decomposition} becomes 
\begin{equation*}
  - \Delta = - D^2  + \frac{(n-1)(n-3)}{4r^2} I + \frac{1}{r^2} O^{-1}\LB O .
\end{equation*}
We define $\Lambda = (n-1)(n-3)  + O^{-1}\LB O$, and finally get
\begin{equation}
  \label{eq=decomposition}
 - \Delta = -D^2 + \frac{1}{r^2} \Lambda.
\end{equation}

\subsection{Exterior scaling}
\label{subsec=exteriorScaling}
Let $\theta\in\xR$. The exterior dilation of an operator $H$ is defined in polar coordinates by:
\[ H_\theta = U(\theta) H  U(\theta)^{-1}\]
where $U(\theta)f(r,\omega) = f(r_\theta,\omega)$, and $r_\theta = r + (r-r_0)e^\theta$. It is easily seen that $D(\theta) = e^{-\theta}D$, and $\frac{\Lambda}{r^2}(\theta) =  \frac{1}{r_\theta^2} \Lambda$. Therefore, if $H = -\eps^2\Delta + V_\eps$ (on the outside of the ball), 
then
\begin{equation}
  \label{eq=expression_of_the_exterior_scaling}
   H(\theta) =  -\eps^2 e^{-2\theta} D^2 + \eps^2 \frac{\Lambda}{r_\theta^2}  + V_\eps(r_\theta,\omega).
\end{equation}
We refer to the appendix for the expression of the symbol of this operator.

This modification of the operator is then extended to complex $\theta$  by analyticity (exterior complex scaling). We will then define resonances to be eigenvalues of $H(\theta)$~: this coincides with the definition in terms of continuation of the resolvent, at least for simple complex scaling (\cf{} \cite{RS78_IV}).

\subsection{Some operators}
\label{sec=operators}
We write down some of the auxiliary operators involved here, for future reference. 
\begin{itemize}
    \item $\fullOp = -\eps^2\Delta + V_\eps$ is the operator we would like to
	study.
    \item $H^{\rZero} = \intOp \oplus \extOp$ is the operator with the same
	symbol, but with a Dirichlet condition on the sphere of radius
	$\rZero$. It decomposes into an interior and an exterior part.
    \item $\rotOp$ is the exterior dilation of $\fullOp$ (outside the sphere
	of radius $\rZero$).
    \item $\extRotOp$ is the exterior dilation of $\extOp$. 
\end{itemize}

\section{Preliminary Agmon-type estimates}
\label{sec=Agmon}
\subsection{The decay of eigenfunctions}
The estimates we present here are in the spirit of Agmon's \cite{Ag82}; we refer to \cite{HS84} and the online course \cite{Hel95} for details. We will follow the last two references, making only slight modifications of the arguments.

The estimates we seek are a way to express the following (informal) statement.
\begin{prpstn}
  \label{prop=informalStatement}
  A Schrödinger operator $-\eps^2 \Delta + V$  may be  approximated by a sum of \emph{harmonic oscillators} ($ - \eps^2 \Delta + c_a(x-a)^2$) located at the minima of $V$.  In particular, the eigenfunctions associated to an eigenvalue coming from an harmonic oscillator at $a$ are concentrated near $a$. 
\end{prpstn}
\begin{rmrk}
  This informal description is accurate when we study the spectrum below $\liminf V$. Let us note however two differences between our case and the usual one. The first is that our $V$ depends on $\eps$. This is what explains the appearance of \emph{exponentially small eigenvalues}. The additional difficulty of our degenerate case (where $\liminf V = 0$) is that the spectra of the harmonic oscillators at the minima is lost in the essential spectrum coming from the behavior of $V$ at infinity. 
\end{rmrk}
We start with the following ``basic'' estimate.
\begin{prpstn}[\cite{Hel95}, prop. 8.2.1]
  If $\Omega$ is a bounded domain in $\R^d$ with $\mathcal{C}^2$ boundary, $V$ is continuous on $\bar{\Omega}$ and $\phi$ is a real valued lipschitzian function on $\bar{\Omega}$, then for any $u\in \mathcal{C}^2(\bar{\Omega}, \R)$ such that $u\on{\partial{\Omega}} = 0$, 
  \begin{equation}
    \label{eq=basicAgmon}
    \begin{aligned}
    \eps^2 &\int \abs{\nabla ( \exp(\phi/\eps) u)}^2 dx + \int ( V - \abs{\nabla \phi }^2) \exp( 2\phi/\eps) u^2 dx \\
    = &\int \exp(2\phi/\eps) (-\eps^2 \Delta u + Vu) u dx.
  \end{aligned}
  \end{equation}
\end{prpstn}
To prove this with a regular $\phi$, just set $v = \exp(\phi/\eps) u$ in the Green--Riemann formula $\int \abs{\grad v}^2 = - \int \Delta v \cdot v$; the general case follows by a regularisation argument.

If we plug a ``good'' $\phi$ into this estimate, and apply it to an eigenfunction $u$, we obtain $L^2$ estimates on the weighted function $u\exp(\phi/\eps)$. This will tell us that $u$ must be small when $\phi$ is big, or in other words that $u$ is localized near the small values of $\phi$.

The ``good'' $\phi$ turns out to be related to a new metric, which takes into account the function $V$. 
\begin{dfntn}
  The \emph{Agmon metric} is defined by $Vdx^2$, where $dx^2$ is the euclidean metric on $\R^d$. In other words,
  \[
  \agmon(x,y ) = \inf \left\{ \int_0^1 \sqrt{V(\gamma(t)) }\abs{\gamma'(t)} dt ; \gamma \in \Gamma(x,y) \right\}
  \]
  where $\Gamma(x,y)$ is the set of $\mathcal{C}^1$ paths joining $x$ to $y$. 
\end{dfntn}
\begin{rmrk}
  In the usual setting, the potential does not depend on $\eps$. Here, $V_\eps$ does vary with $\eps$; however, we define the Agmon metric using only $V= \abs{\nabla F}^2$. We could probably drop the metric entirely and use $F$ instead; we keep it for the sake of intuition and comparison with known results. 
\end{rmrk}
This metric degenerates on the minimas of $V$ (\ie{} the critical points of $F$, \cf{} \eqref{eq=VandF}), and it can be shown that:
\begin{align*}
  \agmon(x,y) &\geq \abs{F(x) - F(y)}, \\
  \nabla_y \agmon(A,y) & \leq V(y),
\end{align*}
for any closed set $A$ and almost every $y$.  Once more, we refer \eg{} to \cite{Hel95} (sec. 8.3) for details. 

The main result of this section is the following rigorous statement in the spirit of proposition \ref{prop=informalStatement}. 
\begin{thrm}[A rough decay estimate]
    \label{thm=decayOfEigenfunctions}
  Let $\Omega$ be a bounded domain, $\op = - \eps^2 \Delta + V_\eps$ in $\Omega$ with Dirichlet boundary condition. Let $\lambda$ be an eigenvalue of $\op$ going to zero (when $\eps \to 0$), and $u$ be a corresponding (normalized) eigenfunction. Let $\minima$ be the set of global minima of $V$ (\ie{} the critical points of $F$).
  Then for any $\delta >0$, there exists $\eps_0$ and a constant $C_\delta$ such that:
  \begin{equation}
    \label{eq=decayOfEigenfunctions-fixedOp}
    \forall \eps<\eps_0, \quad
    \lTwoNrm{ u \exp ( d(x) / \eps ) } + \lTwoNrm{ \nabla { u \exp ( d(x) / \eps ) }} 
    \leq C_\delta \exp ( \delta / \eps),
  \end{equation}
  where $d(x) = \agmon (x, \minima )$.
\end{thrm}

Before we turn to the proof of this result, let us mention here that very similar ideas will be used later when we reconstruct a resolvent from two different parts. Since the operators involved will be a bit different, we delay that discussion (but see the proof of theorem \ref{prpstn=resolventsGluing}, section\ref{sec=putTogether}).

\subsection{Proof of theorem \ref{thm=decayOfEigenfunctions}}
The proof follows closely the one of Theorem 8.4.1  of \cite{Hel95} (the changes come from the dependance of $V$ in $\eps$).

  Let $\deltaTilde$ be a small number (to be fixed later, depending on $\Omega$ and $\delta$).
  We use \eqref{eq=basicAgmon} with $V = V_\eps - \lambda$ and $\phi (\cdot)= (1-\deltaTilde)\agmon(\minima,\cdot)$.
  Since $u$ is an eigenvector, the r.h.s disappears and we get:
  \[
    \eps^2 \int \abs{\nabla ( \exp(\phi/\eps) u)}^2 dx + \int ( V_\eps - \lambda - \abs{\nabla \phi }^2) \exp( 2\phi/\eps) u^2 dx 
    =0.
  \]
  We cut the second integral in two parts, setting $ \Omega_+ = \{x; V \geq \deltaTilde \}$, $\Omega_- = \Omega\setminus \Omega_+$. 
  \begin{align}
      \label{eq=firstStepLHS}
    \eps^2 \int & \abs{\nabla ( \exp(\phi/\eps) u)}^2 dx
    + \int_{\Omega_+} (
        V_\eps - \lambda - \abs{\nabla \phi }^2
      ) \exp( 2\phi/\eps) u^2 dx  \\
    &= -\int_{\Omega_-}  (
        V_\eps - \lambda - \abs{\nabla \phi }^2) \exp( 2\phi/\eps)
      )u^2 dx \notag \\
    &\leq \sup_{\Omega_-}\left( \abs{V_\eps - \lambda - \abs{\nabla \phi}^2 }\right) 
    \int_{\Omega_-} \exp(2\phi/\eps) u^2 dx \notag \\
    &\leq C \int_{\Omega_-} \exp(2\phi/\eps) u^2 dx
    \label{eq=firstStep},
  \end{align}
  where $C$ does not depend on $\eps$ and $\deltaTilde$ (indeed, $\lambda$ goes to zero with $\eps$, and $\abs{V_\eps}\leq \sup_{\Omega} \abs{\grad F}^2 + \eps \sup_\Omega \abs{\Delta F} \leq C$, by compactness).
  We now bound the left-hand side \eqref{eq=firstStepLHS} from below. 
  \begin{align*}
  V_\eps - \lambda - \abs{\nabla \phi}^2 
  &=    V - \frac{1}{2} \eps \Delta F - \lambda - (1-\deltaTilde)^2\abs{ \nabla \agmon(x,\minima)}^2 \\
  &\geq V - \frac{1}{2} \eps \Delta F - \lambda - (1-\deltaTilde)^2 V \\
  &\geq   - \frac{1}{2} \eps \Delta F - \lambda + \deltaTilde(2-\deltaTilde) V.
  \end{align*}
  For $x\in\Omega_+$,   $V\geq \deltaTilde$. By compactness, $\Delta F$ is bounded, so that:
  \[
    V_\eps - \lambda - \abs{\nabla \phi}^2 \geq \deltaTilde^2(2-\deltaTilde) + C(\eps),
  \]
  where $C(\eps)$ goes to zero. For $\epsilon$ sufficiently small, (depending on $\deltaTilde$), 
  \[
    V_\eps - \lambda - \abs{\nabla \phi}^2 \geq \deltaTilde^2.
  \]
  We inject this in $\eqref{eq=firstStepLHS}\leq \eqref{eq=firstStep}$ to obtain:
  \[
    \eps^2 \int 
    \abs{\nabla ( \exp(\phi/\eps) u)}^2 dx
    + \deltaTilde^2 \int_{\Omega_+} \exp( 2\phi/\eps) u^2 dx 
    \leq C \int_{\Omega_-} \exp(2\phi/\eps) u^2 dx
  \]
  We add $\deltaTilde\int_{\Omega_-} \exp(2\phi/\eps) u^2 dx$ on both sides to get
  \[
  \begin{aligned}
    \eps^2 \int &\abs{\nabla ( \exp(\phi/\eps) u)}^2 dx
    + \deltaTilde^2 \int  \exp( 2\phi/\eps) u^2 dx \\
    &\leq (C + \deltaTilde^2) \int_{\Omega_-} \exp(2\phi/\eps) u^2 dx
  \end{aligned}
  \]
  On the r.h.s, we bound the $\phi$ from above, and incorporate $\deltaTilde$ into (a new) $C$:
  \begin{align*}
    \eps^2 \int & \abs{\nabla ( \exp(\phi/\eps) u)}^2 dx
    + \deltaTilde^2 \int  \exp( 2\phi/\eps) u^2 dx \\
    &\leq C \int_{\Omega_- } \exp\left( \frac{2(1-\deltaTilde)}{\eps} \agmon(x,\minima) \right) u^2 dx \\
    &\leq C   \exp\left(\frac{2(1-\deltaTilde)}{\eps} \sup_{\Omega_-}\agmon(x,\minima) \right),
  \end{align*}
  since $u$ is normalized. 
  The continuity of $\agmon$, of $V$ and a compactness argument shows that we can choose $\deltaTilde$ small enough to ensure:
  \[ 2 \agmon(x,\minima) \leq \delta/3 \]
  on $\Omega_-$ (which depends on $\deltaTilde$). In words, the only places where $V$ can be small is on small balls near its minima. The estimate becomes
  \[
    \eps^2 \int \abs{\nabla ( \exp(\phi/\eps) u)}^2 dx
    + \deltaTilde^2 \int  \exp( 2\phi/\eps) u^2 dx 
    \leq C \exp \left( \frac{\delta}{3\eps} \right).
  \]
  It is now easy to obtain the bound we seek. Indeed, we may choose $\deltaTilde$ such that 
  \[
    \deltaTilde \sup_{\Omega} \agmon(x) \leq \delta /3,
  \]
  (here we use strongly the fact that our domain is bounded), and $\eps^{-2} \leq C\exp(\delta/(3\eps))$. Remembering that $\phi = (1-\deltaTilde)\agmon(\cdot,\minima)$, we obtain:
  \[
  \int \exp (2\agmon(x)/\eps) u^2(x) dx \leq \frac{C}{\deltaTilde^2}\exp\left( \frac{ \delta}{\eps} \right).
  \]
  The bound on $\int \abs{\nabla(\exp (\agmon(x)/\eps)u)}^2$ is obtained similarly.


\section{The lower spectrum of the interior operator}
\label{sec=lowerSpectrum}
We are now in a position to describe the bottom of the spectrum of the operator $\intOp$. Let $d_0=\infty,d_1,d_2,\ldots, d_N$ be the critical heights of the potential $F$. 
\begin{thrm}
  \label{thm=spectrumOfTheInteriorOperator}
 There exist  a $d'_0>d_1$, $N+1$ functions $\lambda_0(\eps), \lambda_1(\eps), \ldots, \lambda_N(\eps)$ such that:
 \begin{align*}
     \lambda_0(\eps) &= \bigO(e^{- d'_0/\eps})\\
     \forall 1\leq i \leq N, \quad\lambda_i(\eps) &\logSim \exp\left( - \frac{d_i}{\eps} \right) \\
   \sigma(\intOp)  &= \{ \lambda_i(\eps), 0\leq i \leq N\} \cup S,
 \end{align*}
 where $S\subset [C\eps,\infty)$. 
\end{thrm}
We proceed in two steps: 
\begin{itemize}
  \item we begin by showing the following decomposition: 
    \[ \sigma(\intOp) = S_1 \cup S,\]
    where $S_1$ is a set of at most $N+1$ eigenvalues,  and $S\subset [C\eps,\infty)$.
  \item then we study $S_1$ more precisely, and prove the theorem. 
\end{itemize}

\subsection{A rough division of the spectrum}
This step is mainly a rewriting of known arguments (for the case where the ball does not depend on $\eps$), where we keep track of the dependance on the outside. In particular, we draw heavily on the presentation of \cite{CFKS87}, chapter 11.1 (note that our $\eps$ is their $1/\lambda$, we take $h= \abs{\nabla F}^2$ and $g = \Delta F$, and multiply the whole operator by $\eps^2 = \lambda^{-2}$). 

The main idea is to compare $\intOp$ (in the growing ball) with the operator $\fixedOp$ in a fixed ball $\ball{R}$, which contains all minima of $V$. We first choose an $R'$ such that:
\begin{equation}
    \label{eq=defRPrime}
\inf\left\{ \agmon(x,\minima) , x\in \ball{R'}^c \right\}  =  d'_0 > d_1,
\end{equation}
where $d_1$ is the highest barrier of potential; and $V_\eps\geq C\eps$ when $R'\leq \abs{x}\leq \rZero$.
Then we take $R>R'$ (\eg{} $R = R' + 1$).
We let $d'_i = d_i$ for $1\leq i \leq N$: the $d'_i$ will give the rates of decrease of the exponentially small eigenvalues of $\fixedOp$.

  Following \cite{CFKS87}, we introduce a partition of unity: 
  \begin{equation}
  \label{eq=partitionJ}
    1 = J_0^2 + J_1^2,
  \end{equation}
where $J_0$ is localized outside the fixed ball,  and $J_1$  inside (see figure \ref{fig=partition}).

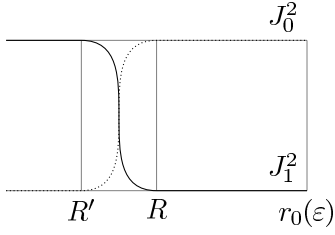
\begin{figure}
  \beginpgfgraphicnamed{partition_of_unity}%
\begin{tikzpicture} 
  \draw[help lines] (-1,0) -- +(4,0) 
                    (-1,2) -- +(4,0)
		    (0,0)  -- +(0,2)
		    (1,0)  -- +(0,2)
		    (3,0)  -- +(0,2);
  \draw[densely dotted] (-1,0) -- (0,0)..controls (1,0) and (0,2) ..(1,2)--(3,2);
  \draw (-1,2)-- (0,2).. controls (1,2) and (0,0) ..(1,0)--(3,0);
  \node[below] at (0,0) {$R'$};
  \node[below] at (1,0) {$R$};
  \node[below] at (3,0) {$\rZero$};
  \node[above left] at (3,2) {$J_0^2$};
  \node[above left] at (3,0) {$J_1^2$};
\end{tikzpicture} 
\endpgfgraphicnamed
\caption{The partition of unity}
\label{fig=partition}
\end{figure}

By the IMS localization formula, 
we have
\begin{equation}
  \label{eq=IMS}
  \intOp = J_0\intOp J_0 + J_1 \intOp J_1  - \eps^2 \sum_{i=1,2} (\nabla J_i)^2.
\end{equation}
Now, the choice of the radius $\rZero$ of the growing ball (\cf{} definition \ref{dfntn=rEps}) ensures that for some $C$, $V_\eps \geq 2C\eps$ on $\supp J_0$. Since $- \Delta$ is positive, we have in terms of quadratic forms:
\begin{equation}
  \label{eq=lowerBoundJ0}
  J_0\left(  -\eps^2 \Delta  + \frac{1}{2} \abs{\nabla F}^2 - \eps \Delta F  \right) J_0 \geq C\eps J_0^2.  
\end{equation}
In words, the operator localized between the fixed ball and the growing ball has a spectrum bounded below by $C\eps$. 

Since the operators are local, we have for any $\phi\in L^2(\largeBall)$: $J_1 \intOp J_1\phi = J_1 \fixedOp J_1\phi$.  Fortunately, the low-lying spectrum of $\fixedOp$ is well-known.
\begin{thrm}
    \label{thrm=interiorSpectrumFixedBall}
  The spectrum of $\fixedOp$ is given by:
  \[ \sigma(\fixedOp) = \{ \mu_0, \ldots, \mu_N \} \cup S,\]
  where 
  \[
  \forall 1\leq i\leq N, \quad 
  \mu_i \logSim \exp\left( - \frac{ d'_i}{\eps} \right) ,
  \]
  $\mu_0(\eps) = \bigO (\eps^\infty)$,
  and $S$ is included in $[C\eps^{3/2}, \infty)$ for some constant $C$.
\end{thrm}
Results in this spirit date back at least to Freidlin and Wentzell's \cite{FW98}; in this special form it can be found in \cite{HN06}. The $3/2$ exponent is not optimal (the statement holds if it is replaced by any quantity which is $o(\eps)$). 

Let $E$ be the span of the $N+1$ first eigenvalues of $\fixedOp$, $P$ the orthogonal projection on $E$, and $K$ the restriction of $\fixedOp$ to $E$. Then 
\[
J_1 \fixedOp J_1  -  J_1 K J_1  \geq C\eps^{3/2} J_1^2.
\]
Let $\tilde{K} =  J_1 K J_1$. If we now plug  \eqref{eq=lowerBoundJ0} and the last equation into  \eqref{eq=IMS}, we get: 
\[
\intOp  - \tilde{K} \geq  C\eps^{3/2}   - C' \eps^2 \geq C''\eps^{3/2},
\]
where $C$ and $C'$ are constants, and $\rank(\tilde{K}) \leq N+1$. 

This is enough to conclude the first step.  Indeed, $\intOp - \tilde{K}$
has no spectrum in the interval $(0,C''\eps^{3/2})$. It is known that a
perturbation by an operator of finite rank can only create as many
eigenvalues as its rank in such an interval (\cf{} \eg{} \cite{Beh78}). Therefore, $\intOp$ has at most $N+1$ eigenvalues in $(0,C\eps^{3/2})$.

\subsection{Approximation of the low-lying eigenvalues}
We precise the approximation of the previous paragraph and show that the first $N+1$ eigenvalues of $\intOp$ are in fact near the ones of $\fixedOp$.

Once more, the intuition is simple: the $N$ eigenvectors of $\fixedOp$ will be shown to be quasimodes (\ie{} approximate eigenvalues and eigenvectors) of $\intOp$. Therefore, a classical result in spectral theory will tell us that near each eigenvalue of $\fixedOp$, there is one for $\intOp$ and this will prove theorem \ref{thm=spectrumOfTheInteriorOperator}.

Let us now be more precise. By theorem \ref{thrm=interiorSpectrumFixedBall}, we know that the $N+1$ exponentially small eigenvalues of $\fixedOp$ are such that: 
\begin{align}
  \label{eq=eigenvaluesOfTheFixedOp}
  \forall 1\leq i \leq N, \quad
  \mu_i(\eps) &\logSim \exp\left( - \frac{d'_i}{\eps} \right)
\end{align}
for any $\delta$.
Let $\phi_i$ be the corresponding normalized eigenfunctions. We would like to consider them as approximate eigenfunctions for $\intOp$. However, $\phi_i$ has no reason to be in the domain of $\intOp$  (because $\Delta\phi_i$, seen as a distribution on $\largeBall$, will have a singular part on  $\partial\fixedBall$). Our approximate eigenfunction will therefore be $\psi_i  = \chi \phi_i$, where $\chi$ is a cutoff function (we may take $\chi = J_1$, where $J_1$ was defined above \eqref{eq=partitionJ})
We will show
\begin{equation}
  \label{eq=psiAreQuasimodes}
  \intOp\tilde{\psi}_i =  \lambda_i\tilde{\psi}_i + \bigO\left( \exp \left(
     - \frac{d'_0 - \delta}{\eps} \right) \right),
\end{equation}
where $\tilde{\psi_i}$ is a normalized version of $\psi_i$.
Once this is shown, the proof is complete: indeed, this implies 
\[
\sigma(\intOp) \cap \left[
  \lambda_i - C_\delta \exp\left( - \frac{d'_0 - \delta}{\eps} \right), 
  \lambda_i + C_\delta \exp\left( - \frac{d'_0 - \delta}{\eps} \right)
		    \right] \neq 0
\]
(\cite{Hel95}, prop. 5.1.4). The asymptotics of $\lambda_i$ ensure that these intervals are disjoint (for $\eps$ small enough), and the error is negligible with respect to the main term $e^{-\lambda_i/\eps}$. Since we already know that, below $C\eps$, the spectrum of $\intOp$ is discrete and contains at most $N+1$ points, it follows that there  each of these $N+1$ eigenvalues must be located in one of these intervals. Thanks to \eqref{eq=eigenvaluesOfTheFixedOp}, this concludes the proof of theorem \ref{thm=spectrumOfTheInteriorOperator}.

We now establish \eqref{eq=psiAreQuasimodes}. We first show the bound for $\psi_i$.
\begin{align*}
  \intOp\psi_i - \lambda_i \psi_i 
  &= \intOp \chi \phi_i - \lambda_i \chi \phi_i \\
  &= \chi \intOp \phi_i - \lambda_i\chi\phi_i + [\intOp,\chi]\phi_i .
\end{align*}
On the support of $\chi$, $\intOp\phi_i$ is well defined and equals $\lambda_i\phi_i$. Therefore
\begin{align*}
\intOp\psi_i - \lambda_i\psi_i 
&= [ \intOp,\chi] \phi_i \\
&= - \eps^2 [\Delta, \chi] \phi_i \\
&= - 2\eps^2 \nabla \chi \nabla \phi_i - \eps^2(\Delta \chi) \phi_i.
\end{align*}
Taking norms, we get
\[
\lTwoNrm{\intOp\psi_i - \lambda_i\psi_i}^2 \leq 4\eps^4\lTwoNrm{\nabla\chi \nabla\phi_i}^2 + 2 \eps^4\lTwoNrm{(\Delta \chi) \phi_i}^2.
\]
We now use the fact that $\phi_i$ is small when we are far from the critical points of $V$, therefore on the support of $\nabla \chi$ and $\Delta \chi$. More precisely, 
\begin{align*}
  \lTwoNrm{\phi_i (\Delta \chi)}
  &\leq \exp\left( -  \inf_{\supp \chi} (d(x))/\eps)\right)
    \lTwoNrm{\phi_i e^{d(x)/\eps}\Delta\chi}^2 \\
  &\leq C \exp\left(-\frac{ \inf_{\supp \chi}d(x) -\delta}{\eps}\right),
\end{align*}
where the second bound follows from  the decay estimate for the fixed operator (equation \eqref{eq=decayOfEigenfunctions-fixedOp}). By a similar argument (using the other part of \eqref{eq=decayOfEigenfunctions-fixedOp} to bound  $\nabla \phi$), we get:
\[
\lTwoNrm{\intOp\psi_i - \lambda_i\psi_i}^2
\leq C_\delta \exp\left(-\frac{ \inf_{\supp \chi}d(x) -\delta}{\eps}\right).
\]
The definition of $R'$ and $d'_0$ (equation \eqref{eq=defRPrime}) implies:
\[
\lTwoNrm{\intOp\psi_i - \lambda_i\psi_i}^2
\leq C_\delta \exp\left(-\frac{ d'_0 -\delta}{\eps}\right),
\]
and the desired bound is proved, for the non-normalized functions $\psi_i$. However, since $\phi_i$ is normalized and localized inside the fixed ball, similar arguments show that $\norm{\psi_i} \geq 1/2$ for small $\eps$. This concludes the proof of \eqref{eq=psiAreQuasimodes}, and theorem \ref{thm=spectrumOfTheInteriorOperator} is proved.


\section{Bounds on the exterior resolvent}
\label{sec=exteriorResolvent}
\subsection{The general strategy}
\label{sec=strategy}
We prove here that the exterior part of the dilated Dirichlet resolvent is regular in the neighborhood of the small eigenvalues. We are interested in a bound on $\extDResolv(\theta, z)$ when $z$ is on a contour around one of the eigenvalues $\lambda_j$. Since $\lambda_j$ is exponentially small, the contour is in a small neighbourhood of $0$, and since $0$ is in the essential spectrum of $H^D_e(\theta)$, the best bound we can hope for is of the type: 
\[
  \norm{ \extDResolv(\theta,z)} \leq \frac{\text{const}}{\lambda_j}.
\]
The following result will be enough for our purpose.
\begin{thrm}
  \label{thrm=boundOnTheExteriorResolvent}
  Let $\lambda_j$ be the exponentially small eigenvalues of the interior operator (\cf{} Theorem \ref{thm=spectrumOfTheInteriorOperator}). Let $\eta>0$. There exists $\theta = i\beta,c_z,C$, independent of $\eps$, such that, if 
  $\abs{z-\lambda_j}\leq c_z\lambda_j$,
  \[
  \norm{ \extDResolv(\theta,z) } \leq \frac{ C_\eta }{ \lambda_j^{1+\eta}},
  \]
\end{thrm}

The main problem to show such a bound is the behaviour at infinity. We
investigate it by using techniques of pseudo-differential operators.
However, these techniques are mainly known when the symbol of the
operator depends smoothly on the parameters (which is not the case here,
since we put a Dirichlet boundary condition on a sphere). Therefore, we
will work separately on the two ``boundaries'' of our domain. Let
$\chi_0, \chi_1$ be a partition of unity, where $\chi_0$ is $1$ on the
ball $\largeBall$ and $\chi_1 = 1$ at infinity (the cut-off functions
will be defined later, \cf{} fig \ref{fig=indicators}). We will define two auxiliary operators $H_0$ and  $H_1$:
\begin{itemize}
    \item The operator ``at infinity'',  $H_1$, will be defined by pseudo-differential operator theory (cf{} section \ref{sec=PDO}), 
  \item We define $H_0$ with a  Dirichlet condition on the sphere, but without degeneracy at infinity, and bound its resolvent (section \ref{sec=theDirichletPart}).
\end{itemize}
Once these two steps are done, we construct  an approximate resolvent by gluing $R_0$ and $R_1$, considering $R = \chi_0 R_0 + \chi_1 R_1$. We finally deduce a bound on the true resolvent (section \ref{sec=putTogether}).

We begin by preliminary estimates on $V_\eps$.


\subsection{Some estimates on \texorpdfstring{$F$ and $V$}{F and V}}
  Let us gather some consequences of the hypotheses on $F$. Recall that $F$ is
  analytic in a region of $\xC^d$ defined by equation
  \eqref{eq=defAnalyticityCone}. Consider the following subset of $\xC$: 
  \[
  \mathcal{R} = \left\{ r , \abs{r}\geq r_0,\arg(r) \leq \tan(2\beta_0).\right\}.
  \]
  The following subset of $\xR^d$ is contained in the analyticity region for
  $F$: 
  \[ 
  \left\{ r\omega = (\omega_1 r, \omega_2 r, \ldots \omega_n r), r\in \xC,
  \omega \in \mathcal{S}_{\xR}^ {n-1}, r\in\mathcal{R}.\right\}.
  \] 
  Therefore, 
  for each $\omega$, $\tilde{V}_{\omega,\eps}: r\mapsto V_\eps(r\omega)$ is 
  analytic in  $\mathcal{R}$.
  The exterior scaled potential $V_\theta(x)$, for $x = r\omega$, coincides with $\tilde{V}_\omega(r_\theta)$, where $r_\theta = r_0 + (r-r_0)e^ {\theta}$.
  We will only consider imaginary $\theta$, and let $\theta = i\beta$.
  
  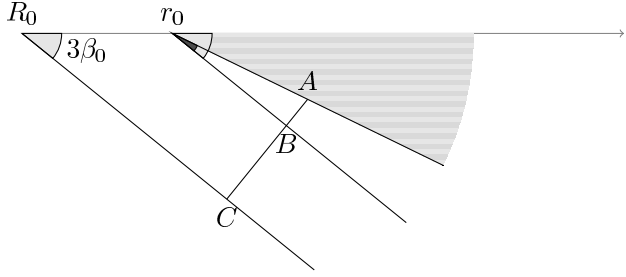
\begin{figure}
    \label{fig=analyticity}
  \beginpgfgraphicnamed{analyticity}%
    \def\Beta{13}          
    \def\angleRadius{15pt} 
  \begin{tikzpicture}
      [angle/.style={fill=black!10},
      rotate=-3*\Beta]
      \draw[help lines,->] (0,0) -- (3*\Beta:8cm);
      \coordinate[label=above:$R_0$] (R zero)    at (0,0);
      \coordinate[label=above:$r_0$] (r zero)    at (3*\Beta:2cm);
      \fill[pattern=horizontal lines light gray] 
        (r zero) -- +(\Beta:4cm) arc (\Beta:3*\Beta:4cm) -- cycle;
      \coordinate[label=above:$A$]   (A) at ($ (r zero) + (\Beta:2cm) $);
      \path (A) |- (r zero) node[coordinate,midway,label=below:$B$] (B) {};
      \path (A) |- (R zero) node[coordinate,midway,label=below:$C$] (C) {};
      \filldraw [angle]
        (0,0)    --   (\angleRadius,0)     arc (0:3*\Beta:\angleRadius)   -- cycle;
      \node[right] at (\Beta:\angleRadius) {$3\beta_0$};
      \filldraw [angle]
        (r zero) -- + (\angleRadius,0)     arc (0:3*\Beta:\angleRadius)   -- cycle;
      \filldraw [angle,fill=black!70]
        (r zero) -- + (0.7*\angleRadius,0) arc (0:\Beta:0.7*\angleRadius) -- cycle;
      \draw (0,0)--(5,0);
      \draw (r zero) -- +(4,0);
      \draw (r zero) -- +(\Beta:4);
      \draw (A) -- (C);
  \end{tikzpicture} 
  \endpgfgraphicnamed
  \caption{The analyticity region of the map $\tilde{V}_\omega$ and the
  relevant  Cauchy contours.}
  The function $V$ is analytic in the whole sector of angle $3\beta_0$ (light grey angle). Therefore, the distance between a generic point in the colored sector (where the $r_\theta$ live, for $\theta\leq 2\beta_0$) and the non analyticity region is at least $AC = AB + BC$. Since $BC = (r_0 - R_0) \tan(3\beta_0)$, and $AB = (r-r_0) \sin(\beta_0)$, $AC \geq (r-R_0) \sin(\beta_0)$. So the circle centered in $r_\theta$ with radius $(1/2)(r-R_0)\sin(\beta_0)$ is entirely contained in the analyticity region. 
\end{figure}

  \begin{prpstn}
      \label{prpstn=TaylorBoundsOnRTheta}
  The following development holds, for small $\beta = \xIm(\theta)$:
  \begin{equation}
    \label{eq=firstOrderBnd}
    \forall r\geq r_0(\eps), \quad
    V(x_\theta) =  V(r_\theta ,\omega) = V(r,\omega)(1 + \bigO(\beta)),
  \end{equation}
  where the $\bigO(\beta)$ takes complex values,  but  \emph{does not depend on $\eps,r,\omega$}.
  Moreover, on the region $V(r) \leq 2 \lambda_j$, 
  \begin{equation}
    \label{eq=secondOrderBnd}
    V(x_\theta) = V(r,\omega) + i \beta (r- r_0)\devPart{}{r}V(r,\omega)(1 + \bigO(\beta)),
  \end{equation}
  \end{prpstn}

  \begin{rmrk}
    \label{rmrk=aboutRAndRZero}
     Note that, given the growth rates of $V$, and the fact that $\rZero(\eps)$ is polynomial in $\eps$ and $\lambda_j(\eps)$ exponentially small, $\rZero$ is much smaller than $r$ if $V(r) = \lambda_j$. We will take $\eps$ small enough so that:
     \[V(r) \leq 2\lambda_j \quad\implies\quad r-\rZero \geq \frac{1}{2}r. \]
  \end{rmrk}

  \begin{proof}
    These bounds are given by the Taylor approximation of $V(r_\theta)$ for small $\theta$. The strong hypotheses on $V$ guarantee that, for small $\beta$ \emph{independent of $x$}, the first terms of the development are the main ones. To see it, we first prove 
\begin{lmm}
  \label{lmm=upperBoundsOnV}
  For $\alpha<\beta_0$, 
  $\abs{\tldV'(r_{i\alpha})} \leq Cr^{-\gamma -1}$, and $
  \abs{\tldV''(r_{i\alpha})} \leq Cr^{-\gamma -2}$, where the constants do not
  depend on $\omega, \eps$.
\end{lmm}
This follows from the estimates on $V$ and analyticity. Indeed, $\tldV$ is
analytic in a conical region of angle $3\beta_0$. Therefore, for any
$r_\theta$, with $\theta = i\alpha$ and $\alpha<2\beta_0$, the circle centered
in $r_\theta$ with radius $(r-R_0)\sin(\beta_0)/2$ is contained in the cone (\cf{}
figure \ref{fig=analyticity}). Apply Cauchy's formula on this circle:
\[
\tldV'(r_\theta)
= \frac{1}{2i\pi} \int_{\text{circle}} \frac{\tldV(z)}{z-r}.
\] 
On this circle,
$\abs{z} \geq r/2$ so 
$\tldV \leq 2^ \gamma C_V r^{-\gamma}$, and $\abs{z-r}  =
(1/2)(r-R_0)\sin(\beta_0)$.
For $\eps$ small enough, since $R_0$ is fixed and
$r$ is bigger than $r_0(\eps)$ (which is larger and larger), we have $r - R_0
\geq r/2$. Therefore $\tldV'(r_\theta)\leq C r^ {-\gamma-1}$, and the first claim
is proved (for $\alpha<2\beta_0$).
 We now repeat the reasoning with $\tldV'$ instead of
$\tldV$: since we know how to bound $\tldV'$ on the cone of angle $2\beta_0$, we
deduce bounds on $\tldV''$ on the smaller cone of angle $\beta_0$.This concludes the proof of lemma
\ref{lmm=upperBoundsOnV}.

  Let us go back to the proof of \eqref{eq=firstOrderBnd}. The first-order Taylor expansion of $V(r_\theta)$ reads:
    \[
    \tldV(r_\theta)  = \tldV(r) + \int_{0}^{\beta} i(r-r_0)e^ {i\alpha} \tldV'(r_{i\alpha}) d\alpha.
    \]
    Using the upper bound on $\tldV'$ (previous lemma) and the lower bound on
    $\tldV$ (hypothesis), we see that $\abs{r\tldV'(r_{i\alpha}) }\leq C\abs{V(r)}$ (where $C$ does not depend on $\eps,\beta$). This shows \eqref{eq=firstOrderBnd}.

    The second bound follows from the Taylor expansion up to order $2$, using
    lemma \ref{lmm=upperBoundsOnV} and remark \ref{rmrk=aboutRAndRZero} to
    bound $\tldV''$ from above, and hypothesis \ref{hpthss=decay} to bound
    $\tldV'$  from below.  
  \end{proof}

  We also need to bound partial derivatives with respect to the cartesian
  coordinate $x_i$. 
  \begin{prpstn}
      \label{prpstn=boundsOnCartesianDerivatives}
      Each partial derivative of $V$ is smaller by a factor of $1/r$. More
      precisely, 
      \[
      \abs{\partial_x^\alpha V_\theta (x) } \aleq r(x)^{-\abs{\alpha}}
      \abs{V_\theta(x)}
      \]
      when $r\geq \rZero$.
  \end{prpstn}
  \begin{proof}
     We use the same ideas as in the proof of lemma \ref{lmm=upperBoundsOnV}.
     Let $x$ be such that $r(x) \geq r_0$.
     Suppose we freeze the coefficients $x_2, \ldots x_d$, and consider the
     map:
     \[ \phi: x_1 \mapsto V_\theta(x_1, \ldots x_d).\]
     This function $\phi$ has an analytic continuation to a region that
     contains a circle of radius of order $r(x)$. Using the Cauchy formula on
     this circle, and the a priori upper and lower bounds on the analytic continuation
     of $V$, we prove the claim.
  \end{proof}


\subsection{The resolvent ``at infinity'' --- symbol bounds}
\label{sec=PDO}
\subsubsection{The strategy}
Following the strategy outlined in section \ref{sec=strategy}, we start
by defining the operator $H_1$. We obtain $H_1$ by modifying the original operator in two ways. First,  $V_\eps$ is replaced by a function $V_{1,\eps}$ such that:
\begin{itemize}
  \item $V_{1,\eps} = V_\eps$ on the support of $\chi_1$.
  \item $V_{1,\eps}$ is smooth and greater than $C\epsilon$ inside the ball $\largeBall$.
\end{itemize}
The second condition may be imposed since by definition of the radius $R(\epsilon)$, $V \geq C\epsilon$ on the boundary of the ball. 
\begin{rmrk}
    For notational convenience, and since the problematic behaviour of the operator comes from the part where $V_\eps = V_{1,\eps}$, we will write $V_1$, or even $V$, instead of $V_{1,\eps}$.
\end{rmrk}

The other modification is in the kinetic term. To define it, we let $h(x,\xi ;
\epsilon, \theta)$ be the symbol of the exterior-scaled Dirichlet operator (an
explicit expression is given in equation \eqref{eq=symbolDilatedLaplacian}). We
modify $h$ near the boundary to make it smooth: let $\chiSmooth$ be a smooth
cutoff function supported near $\largeBall$ and with value $1$ on the ball
(\cf{} figure \ref{fig=indicators} for a precise definition), we define the
smoothed symbol
\[
h_s(x,\xi;\eps,\theta) = \chi(x) \sigma(- \Delta) + (1- \chi(x)) \sigma( - \Delta_\theta). 
\]
Adding the scaled potential, we obtain: 
\[ 
  h_1(x,\xi;\eps,\theta) = h_s(x,\xi) + V_\theta(x).
\]
This function is, for each $x$, polynomial in $\xi$ (of order $2$). Therefore,
it defines by quantification (\cf{}  section \ref{sec=pseudors} in the appendix) an operator $H_1$.

The main idea is to construct an approximate resolvent by the following formula:
\[
 (H_1 - z)^{-1} \approx \pdo\left( \frac{1}{h_1 - z} \right).
\]
\begin{rmrk}
    This idea is behind the classical construction of a \emph{parametrix}
    (\cf{} appendix). However we need here an explicit $L^2$ control (not
    only smoothing), therefore we will use explicit expressions of the
    remainder, given in terms of oscillatory integrals (\cf{} theorem
    \ref{thrm=symbolicCalculus}, in the appendix).
\end{rmrk}
To apply regularity results from \PDO{} theory, we need estimates on the symbol and its derivatives.
\begin{prpstn}
  \label{prpstn=boundsOnTheSymbol}
  For some $\theta = i\beta$, there exists constants $c,C$ and $c_z$ (independent of $\eps,x,\xi$) such that, when  $z$ is on a small circle around $\lambda = \lambda_j(\eps)$ ($z = \lambda_j(1 + c_z e^{i\omega})$)
  \begin{align}
    \label{eq=TheLowerBound}
    \forall \eps,x,\xi, \quad
    \abs{ h_1(x,\xi) - z} &\geq c \max\left( M(x,\xi) ,\lambda\right)\\
    \label{eq=boundsOnDerivatives}
    \abs{\dXAlphaDXiBeta h_1(x,\xi)}
    & \leq 
    \left(\frac{r_0}{r}\right)^{\abs{\alpha}}
        \left( \ind{\beta = 0}\left( V(r) \ind{r>r_0} + \ind{r<r_0}\right) + \eps^2 \max(\frac{1}{r},\abs{\xi})^{2 - \abs{\beta}} \right),\\
    \label{eq=inverseBound}
    \abs{\dXAlphaDXiBeta \frac{1}{h_1(x,\xi) -  \lambda}}
    &\leq C\sum_{n=0}^{\abs{\alpha} +  \abs{\beta}}
    \left( \frac{r_0}{r} \right)^{\abs{\alpha}}
        \frac{M(x,\xi)^{n - \abs{\beta}/2}}{\max\left(M(x,\xi), \lambda\right)^{1+n}}
  \end{align}
  where $M(x,\xi) = \max( V_1(x), \eps^2 \abs{\xi}^2)$.
\end{prpstn}
This is proved in the following sections (\ref{sec=boundsOnTheSymbol}, \ref{sec=boundsOnDerivatives}).

\bigskip

The next step is to use the pseudo-differential theory to obtain operator bounds:
\begin{prpstn}
  \label{prpstn=boundsOnOperators}
  \begin{enumerate}
    \item The approximate resolvent $G = \pdo\left(1 / (h_1 - z) \right)$ is ``almost'' bounded by $1/\lambda$: for all $\eta>0$, 
      \begin{equation}
	\label{eq=boundOnG}
	\norm{G} \aleq C_\eta\lambda_j^{-(1+\eta)}.
      \end{equation}
    \item The same estimate holds for the real resolvent $(H_1 - z)^{-1}$.
  \end{enumerate}
\end{prpstn}
This is proved in 
section \ref{sec=fromPDOtoResolvent}.

\subsubsection{The lower bound}
\label{sec=boundsOnTheSymbol}
In this section we prove proposition \ref{prpstn=boundsOnTheSymbol}. Note that it is enough to show a lower bound on $\abs{h_1(x,\xi) - \lambda}$ (from which the desired bound follows, up to a change of $c_z$ and $c$).
  Recall that
  \begin{align}
      \notag
  h_1 - \lambda &= h_s(x,\xi) + V_1(x) - \lambda \\
  \label{eq=decompositionSymbole}
  &= \eps^2\xi^2\left( \chi(x) + e^{-2\theta}(1 - \chi(x))\right)\\
  \notag
   &\quad + \eps^2(1 - \chi(x))\left( \frac{r^2}{r_\theta^2} - e^{-2\theta}\right) \left( \abs{\xi}^2 + \sigma(D^2) \right)\\
   \notag
   &\quad + V_1(x_\theta) - \lambda.
 \end{align}
  where $\chi(x)$ is $1$ for $r\leq r_0$ and $0$ at infinity.
 
  We use different arguments for different regions of $(x,\xi)$. 
  Let us begin by an informal explaination before we go into details. In the ``interesting'' regions ($x$ sufficiently large), $h_1 - \lambda$ should behave in first approximation like its real part, which looks like
  \[
    e^{-2\theta}\eps^2\xi^2 + V_1(x) - \lambda.
  \]
    So when $V$ is large  enough with respect to $\lambda_j$, we can use this real part and positivity to get the desired bounds.
    
    When $V$ is approximately $\lambda$, or even smaller,  the real part will not give us the bound.
    Therefore, we multiply by  $e^{2\theta}$ (to move the kinetic part back to $\xR$), and bound the imaginary part of (approximately) $e^{2\theta}(V_1 - \lambda)$. The result then follows from the development \eqref{eq=secondOrderBnd} of $V$.

  Note that we keep the $\eps^2 \xi^2$ as a term in the maximum, so as to deal with the ``large $\xi$'' regions when we consider derivatives later on, but most of the trouble comes from the potential part.

  We state here two results we will need in the proof. The first one concerns the symbol of $D^2$.
\begin{prpstn}
  The derivatives of the symbol of $D^2$  admits the following bounds:
  \begin{equation}
    \label{eq=boundOnDTwo}
    \abs{\partial_x^\alpha\partial_\xi^\beta \sigma(D^2)}
    \leq C_n\frac{1}{r^{\abs{\alpha}}} \max\left(\frac{1}{r}, \abs{\xi} \right)^{2-\abs{\beta}},
  \end{equation}
  for all multi-indices $\alpha$, $\beta$. (The derivatives are $0$ if $\abs{\beta}\geq 3$).
\end{prpstn}
The expression of $\sigma(D^2)$ is shown in the appendix. The bounds come from the  homogeneity in $x$ and ``polynomialness'' in $\xi$. 

The second result we need is the following elementary lemma, on the sum of ``almost real positive'' numbers.
\begin{lmm}
  \label{lmm=simpleLemma}
  If $a,b$ are two complex numbers with arguments in $-\pi/4, \pi/4$, then $\abs{a + b} \geq \max( \abs{a}, \abs{b})$.
\end{lmm}

We are ready to tackle the first case, when $V$ is larger than $\lambda$. We need a safety margin, and we define:
\begin{equation}
  \label{eq=defSecurity}
  \safetyCst = \frac{c_V}{12 C_V}
\end{equation}
\paragraph{\bfseries First case: $V(r) \geq (1 + \safetyCst) \lambda$.}
   Let us slightly rewrite $h_1$:
   \begin{align}
     \label{eq=decompositionHOne}
  h_1 - \lambda &= \eps^2\xi^2 \times \left( \left( \chi(x) + \frac{r^2}{r_\theta^2}(1 - \chi(x))\right)\right)\\
  \notag
  &\quad  +\eps^2\left((1 - \chi(x))\left( \frac{r^2}{r_\theta^2} - e^{-2\theta}\right)  \sigma(D^2)\right)\\
   \notag
   &\quad + V_1(x_\theta) - \lambda \\
   \notag
   &= \eps^2 \xi^2 K_1 + K_2 + P.
  \end{align}
  The first kinetic factor $K_1$  is a convex combination of $1$ and the
  complex number $\frac{r^2}{r_\theta^2}$, the latter having  small argument
  (for small beta, independent of $\eps,r_0,r$), and a norm bigger than $1$:
  therefore, $\abs{\arg(\eps^2\xi^2)} < \pi/4$, and $\abs{\eps 2\xi^2 K_1}
  \geq \eps^2 \xi^2$.
  
  For the potential term $P$ we use the development \eqref{eq=firstOrderBnd}. 
  So
  \begin{align*}
    V(x_\theta) -\lambda &= V(r)(1 + \bigO(\beta)) - \lambda \\
    &= (V(r) - \lambda) (1 + \bigO(\beta)).
  \end{align*}
  Since $\lambda\leq \frac{1}{1+\safetyCst} V(r)$, $V(r) - \lambda \geq
  \frac{\safetyCst}{1+\safetyCst} V(r)$. Therefore, for $\beta$ small enough,
  and for  some constant $c$,  $\abs{V(x_\theta) - z} \geq cV(r)$ and $\abs{ \arg(V(x_\theta) - z) } < \pi/4$.

  Finally, $\abs{K_2}$ is small, in modulus, with respect to one of the two other terms. 
  Indeed, equation \eqref{eq=boundOnDTwo} entails:
  \begin{equation}
      \label{eq=controlOfK_2}
    \abs{K_2} \leq 5C_N\beta \eps^2 \max\left(\frac{1}{r},\abs{\xi}\right)^2
    \leq \left|
      \begin{array}{ll}
	c\beta V(r)               &\text{ if } \frac{1}{r}\geq \abs{\xi}, \\
	c\beta \eps^2 \abs{\xi}^2 &\text{ if } \abs{\xi}\geq \frac{1}{r}.
      \end{array}
      \right.
  \end{equation}]]]
  This, combined with lemma \ref{lmm=simpleLemma}, shows  that $\abs{h_1 -\lambda} \geq c' \max(\eps^2\abs{\xi}^2, \abs{V(x)}, \abs{\lambda})$, as announced.

  \bigskip

  This ``positivity'' argument  still works in a slightly different setting.
  Indeed, if $V(r) \leq (1 + \safetyCst) \lambda$ but $\eps^2\abs{\xi}^2 \geq
  \lambda$, then $\abs{P} = \abs{V(x_\theta) - \lambda } \leq
  \safetyCst\lambda(1 + \bigO(\beta))\leq 2\safetyCst \eps^2 \xi^2$ (make
  $\bigO(\beta)$ smaller than $1$); and $\abs{K_2} \leq c\beta
  \eps^2\abs{\xi}^2$ (thanks to eq. \eqref{eq=controlOfK_2}). So
  \begin{equation}
    \label{eq=positivityForSmallVLargeXi}
  \abs{ h_1 - \lambda }  = 
  \abs{ \eps^2 \xi^2 K_1  + K_2 + P}
  \geq \frac{1}{2}\eps^2 \xi^2 -  c\beta \eps^2 \xi^2 -  2\safetyCst
  \eps^2\xi^2 \geq (\frac{1}{2} - c\beta - 2\safetyCst)\max(\eps^2\abs{\xi}^2, \lambda),
\end{equation}
and the lower bound \eqref{eq=TheLowerBound} holds (since $\safetyCst$ and
$\beta$ may be taken small).

  \paragraph{\bfseries Second case: $V(r)< (1 + \safetyCst ) \lambda$.}
  Note that this in this region, $r$ is much bigger than $r_0$, therefore
  $\chi$ is $0$. We will  even take $\eps$ small enough so that
  \begin{equation}
    \label{eq=rIsSmall}
    \frac{r_0(\eps)}{r} \leq \frac{\safetyCst}{2(C_N+1)C_V}.
  \end{equation}
  Let us multiply the symbol by $e^{2\theta}$:
  \begin{equation}
    \label{eq=rotatedSymbol}
    e^{2\theta} (h_1 - \lambda)  
    = \eps^2\abs{\xi}^2 + \eps^2\left( \frac{r^2}{r_\theta^2}e^{2\theta} - 1 \right) \left( \abs{\xi}^2 + \sigma(D^2) \right)
      + e^{2\theta}\left( V_1(x_\theta ) - \lambda\right).
  \end{equation}
   We develop the last product for small $\beta$ (recall $\theta = i\beta$). Since $ e^{2\theta} = 1 + 2i\beta(1+ \bigO(\beta))$, 
    \begin{align*}
      e^{2\theta}(V(x_\theta) - \lambda)
      &=  V(x_\theta) - \lambda + 2i\beta(1 + \bigO(\beta)(V(x_\theta) - \lambda).
    \end{align*}
    We develop the first $V(x_\theta)$ to the second order
    (using \eqref{eq=secondOrderBnd}, which holds in this region) and the other
    to the first order (eq.~\eqref{eq=firstOrderBnd}). This yields
    \begin{align}
    e^{2\theta}(V(x_\theta) - \lambda) 
    &= V(r) - \lambda + i\beta(r-r_0)V'(r)\left(1 + \bigO(\beta)\right) \notag \\
    &\quad  + 2i\beta(V(r) - \lambda)(1 + \bigO(\beta))
    \notag \\
\label{eq=devPotentialTerm}
    &= V(r) - \lambda + i\beta\left( (r-r_0)V'(r) + 2(V(r) - \lambda) \right)
      \left(1 + \bigO(\beta)\right).
   \end{align}
   The correction term in the kinetic part  can be shown to have the following development:
   \begin{equation}
     \label{eq=devCorrection}
     \left(\frac{r^2}{r_\theta^2}e^{2\theta} -1\right) = \frac{2ir_0}{r} \beta(1 + \bigO(\beta)).
   \end{equation}

   We return  to $h_1 - \lambda$, and focus on its imaginary part. We use the notation $z_1\equiv z_2$ if $\xIm(z_1) = \xIm(z_2)$. Plug the developments  \eqref{eq=devPotentialTerm} and \eqref{eq=devCorrection} into \eqref{eq=rotatedSymbol}, and dismiss real parts:
   \begin{align}
     \notag
    e^{2\theta}\left( h_1 - \lambda\right)
    &\equiv  2i\beta\frac{r_0}{r} \eps^2 \left( \abs{\xi}^2 + \sigma(D^2) \right)
                                         \left(1 + \bigO(\beta)\right) \\
    &\quad + i\beta\left( (r-r_0) V'(r) + 2(V(r) - \lambda)\right) \left(1 + \bigO(\beta)\right) 
    \notag\\
    \label{eq=almostDone}
    &\equiv i\beta\left(1+\bigO(\beta)\right)\left(%
	\frac{r_0}{r}\eps^2\left(\abs{\xi}^2 + \sigma(D^2) \right) 
        + (r-r_0)V'(r) + 2(V(r) - \lambda)
      \right)
  \end{align}
 Now, $\abs{\abs{\xi}^2 + \sigma(D^2)} \leq (1+C_N)\max\left(\abs{\xi}^2, \frac{1}{r^2}\right)$ (\cf{} \eqref{eq=boundOnDTwo}). Since on the one hand, $r^{-2}\leq r^{- \gamma}\leq C_V V(r)\leq 2C_V\lambda$, and on the other hand we may suppose $\eps^2\abs{\xi}^2 \leq \lambda$ (\cf{} the discussion that leads to eq. \eqref{eq=positivityForSmallVLargeXi}), we obtain thanks to eq. \eqref{eq=rIsSmall}: 
  \[
    \abs{ \frac{r_0}{r}\eps^2 (\abs{\xi}^2 + \sigma(D^2))} \leq \safetyCst\lambda.
  \]
  We claim that the second term in \eqref{eq=almostDone} is bounded below:
  \begin{equation}
    \label{eq=nonTrapping}
    \abs{ (r-r_0)V'(r)  + 2(V(r) - \lambda) }\geq \safetyCst \lambda
  \end{equation}
  Indeed, if $V\geq \lambda/2$, $2(V(r) - \lambda)\leq 2\safetyCst\lambda$, and 
  $(r-r_0)V'(r) \leq \frac{r}{2}V'(r) \leq \frac{c_V}{4C_V}\lambda \leq
  3\safetyCst\lambda$ (using the
  negativity of $V'$, the bounds on $V$ and the definition \eqref{eq=defSecurity} of
  $\safetyCst$).
  This implies \eqref{eq=nonTrapping}. If $V\leq \lambda/2$, the $V(r) - \lambda$ term suffices to show the bound (the  $V'$ term being negative).

  Getting back to \eqref{eq=almostDone}, we obtain
  \begin{align*}
    \abs{h_1 - \lambda)} &\geq \xIm( e^{2\theta} (h_1 - \lambda) ) \\
    & \geq \safetyCst\lambda(1 + \bigO(\beta))
  \end{align*}
  Since in this case, $\lambda \geq \eps^2 \abs{\xi^2}$, the proof of \eqref{eq=TheLowerBound} is finally complete.

  \begin{rmrk}
      \label{rmrk=aboutNonTrapping}
      Equation \eqref{eq=nonTrapping} is a kind of ``non-trapping''
      condition. It says that $V'(r)$ is ``negative enough'' with respect
      to $V$. Informally speaking, a classical particle in that potential
      should escape to infinity (and not get trapped).
  \end{rmrk}


\subsubsection{Bounds on derivatives}
\label{sec=boundsOnDerivatives}
We have seen in detail how to bound the symbol from below. Bounding the derivatives is then mainly a technical problem, and uses the same ideas as before. Therefore, we only give a brief outline of the proofs. 

Recall the decomposition \eqref{eq=decompositionSymbole} of $h(x,\xi)$. We have already seen the behaviour of the derivatives of $\sigma(D^2)$ (\cf{} \eqref{eq=boundOnDTwo}) and $V$. The $x$-derivatives of $\chiSmooth$ satisfy: 
\[ 
  \abs{\dXAlpha (\chiSmooth)} \leq C_\alpha \ind{r_0,r_0 +1}(x).
\]
Using the explicit expression of $r_\theta$, one can see that there exists constants such that
\[
  \abs{\dXAlpha\left( \frac{r^2}{r_\theta^2} - e^{2\theta}\right)}
  \leq C_\alpha \left(\frac{r_0}{r}\right)^{1+\abs{\alpha}} .
\]
The effect of $x$-derivatives on $V$ has already been seen: intuitively, each derivative gains a factor of $1/r$ (\cf{}
proposition \ref{prpstn=boundsOnCartesianDerivatives})

The bounds on the $\xi$-derivatives are even simpler, due to the polynomial character of the symbol. All these estimates imply the bound \eqref{eq=boundsOnDerivatives} on $\dXAlphaDXiBeta h_1$.

\bigskip

To bound the derivatives of $g = 1/(h-z)$, remark that $\dXAlphaDXiBeta g$ is a sum of terms of the following type:
\[ 
   \frac{%
   \prod_{j=1}^k
    \left(\partial_x^{\alpha_j} \partial_\xi^{\beta_j} h \right)^{n_j}
    }{%
    (h- z)^{1 + \sum n_j}
    },
\]
where $n_j\in\xN$, $\alpha_j$ and $\beta_j$ are multiindices, and $\sum_j
n_j\alpha_j = \alpha$, $\sum_j n_j\beta_j = \beta$. We may now use the bound
\eqref{eq=boundsOnDerivatives} on each term. Denote by $n$ the sum $\sum n_i 
\leq \abs{\alpha} + \abs{\beta}$.  All the $(r_0/r)^{\alpha_j}$
terms coming from \eqref{eq=boundsOnDerivatives} recombine to give
$(r_0/r)^{\abs{\alpha}}$, and the same kind of arguments on the $\beta$
derivatives show the bound \eqref{eq=inverseBound}.



\subsection{The resolvent at infinity --- operator bounds}
\label{sec=fromPDOtoResolvent}
\subsubsection{An estimate on \texorpdfstring{$\pdo\left(1/(h_1 -z)\right)$}{G}}
We now turn the symbol bounds of the previous section into bounds for operators in $L^2$. We begin by proving the first item of proposition \ref{prpstn=boundsOnOperators}, namely the bound on $G = \pdo\left(1/(h_1 - z\right)$.

Let $\chi$ be a (non negative with positive $L^2$ norm) bump function, in the
product form $\chi =\chi_x\chi_\xi$, where $\abs{x},\abs{\xi}$ are less than
$1$ on $\supp \chi$. 
\begin{prpstn}
    \label{prpstn=gSatisfiesCV}
    The symbol $g$ satisfies the following bounds.
    \begin{equation}
        \label{eq=gSatisfiesCV}
      \int \abs{ \dXAlphaDXiBeta g}^2 \chi(x-k,\xi - l) dxd\xi
      \aleq \begin{cases}
	  \eps^{-2d} \lambda^{-2} &
	  \text{ if } \abs{\beta}\leq \floor{d/2},\\
	  \eps^{-2d}\lambda^{ - 2 - \abs{\beta} + d/2} & 
	  \text{ if } \floor{d/2}\leq \abs{\beta}\leq \floor{d/2}+1
	\end{cases}
    \end{equation}
\end{prpstn}
%
\begin{proof}
     (Note that $\chi$ is choosed independently of $\eps$, therefore the theorem applies uniformly for every epsilon). We use the bound \eqref{eq=inverseBound} on the derivatives of $g = 1/(h-\lambda)$. Let $Q$ be one of the terms in this bound:
    \[ Q = \frac{M(x)^{n-\abs{\beta}/2}}{\max(M(x,\xi), \lambda)^{1+n}}.\]
    We need to control the following quantity:
    \[ \int _{\xR^{2n}} \abs{Q}^2 \chi(x-k,\xi-l) dx d\xi.\]
    We integrate on two different regions and consider:
    \begin{align}
        \label{eq=AOne}
        A_1 &= \int \abs{Q}^2 \ind{V\leq\eps^2\xi^2} \chi(x-k,\xi-l)dx d\xi  \\
        \label{eq=ATwo}
        A_2 &= \int \abs{Q}^2 \ind{V>\eps^2\xi^2}\chi(x-k,\xi-l) dx d\xi 
    \end{align}
    Let us consider $A_1$ first. On the region of integration, $M(x,\xi) = \eps^2\xi^2$, and we replace the max in the denominator by a sum:
    \[
    A_1 \aleq \int_{x,\xi}
       \frac{(\eps^2\xi^2)^{2n-\abs{\beta}}}
            {(\lambda +\eps^2\xi^2)^{2+2n}} \chi
      dx d\xi.
    \]
    We carry out the integration w.r.t. $x$ (on a bounded set, independent of $\eps$), which leaves us with:
    \[
    A_1 \aleq \int_\xi  \frac{(\eps^2\xi^2)^{2n-\abs{\beta}}}{(\lambda +\eps^2\xi^2)^{2+2n}}\chi_\xi(l-\xi) d\xi.
    \]
    When $\xi$ is large, the integrand is small, so it is enough to consider the case $l=0$. 
    In that case, using polar coordinates for $\xi$, we get:
    \[
    A_1 
    \aleq \int_{\abs{\xi}\leq 1} \frac{\eps^2\xi^2)^{2n-\abs{\beta}}}{(\lambda + \eps^2\xi^2)^{2 + 2n}}
    = \int_0^1 \frac{(\eps^2 r^2)^{2n-\abs{\beta}}}{(\lambda + \eps^2r^2)^{2 + 2n}} r^{d-1}dr.
    \]
    If $\abs{\beta}\leq \floor{d/2}$, we bound $r^{-2\abs{\beta} +d -1}$ by $r^{-1}$ in the numerator. Then we change variables and let $u=r\eps\lambda^{-1/2}$. An easy computation then shows that:
    \[A_1\aleq \eps^{2-2\abs{\beta}} \lambda^{-2} \int_0^\infty \frac{u^{4n - 3}}{(1+u^2)^{2+2n}} du.\]
    Since $n\geq 1$ (there is at least one derivative), the integral is finite and $A_1$ is bounded by the r.h.s. of \eqref{eq=gSatisfiesCV}.

    When $\beta\in[ \floor{\beta/2}, \floor{\beta/2}+ 1]$, the same change of variables leads to
    \[
    A_1 \aleq \eps^{2 - 2\abs{\beta}}\lambda^{-2 - \abs{\beta} + d/2}
    \int_0^\infty \frac{ u^{4n - 2\abs{\beta} + d - 1}}{(1+ u^2)^{2+2n}} du.
    \]
    The power of $u$ in the numerator is between $4n-3$ and $4n+1$, therefore the integral is finite and $A_1$ is once more bounded by the r.h.s. of \eqref{eq=gSatisfiesCV}

    \bigskip

    To bound $A_2$, we use the fact that the set of $\xi$ s.t. $\eps^2\xi^2\leq V$ has volume at most $\eps^{-d}V^{d/2}$. Therefore, for each $x$, 
    \[
      \int_\xi \abs{Q}^2 \ind{V>\eps^2 \xi^2} \chi_x(x-k)\chi_\xi(\xi-l) d\xi
      \leq \frac{V^{2n - \abs{\beta} + d/2}}{\max(\lambda, V)^{1+2n}}.
    \]
    Since $\abs{\beta}\leq d/2 +1$, the r.h.s. is bounded (for any
    $V(x)$) by $\lambda^{-2}$. Since the integration in $x$ is on a
    bounded volume, $A_2$ is bounded by $C/\lambda^2$. This ends the
    proof of proposition \ref{prpstn=gSatisfiesCV} 
\end{proof}
The estimates of proposition \ref{prpstn=gSatisfiesCV}, for the classical
derivatives of $g$, entail similar ones for fractional derivatives:
\begin{prpstn}
  \label{prpstn=gSatisfiesFractionalCV}
  Let $s$ and $s'$ be  real numbers, $d/2<s< \floor{d/2}+1$, $s'=\floor{d/2} + 1$.
  The symbol $g$ satisfies the following: 
  \begin{equation}
    \label{eq=gSatisfiesFractionalBounds}
    \int \left|
      (1 - \Delta_x)^{s'/2} (1 - \Delta_\xi)^{s/2} \left(
        g(x,\xi) \chi(x-k,\xi - l)
      \right)
    \right|^2
    dxd\xi
    \aleq \lambda^{-2(1+s - d/2)}.
  \end{equation}
  uniformly in $k,l$.
\end{prpstn}
With this symbol control, we can apply theorem \ref{thm=CalderonVaillancourt}
in the appendix, and show the first part of
proposition \ref{prpstn=boundsOnOperators}, with $\eta = s-d/2$. Note that
$\eta$ can therefore be made arbitrarily small.
\begin{proof}
    To prove proposition \ref{prpstn=gSatisfiesFractionalCV}, we need to interpolate the bounds for integer derivatives to obtain those for fractional derivatives.
    \begin{lmm}
        For $u\in \mathcal{S'}(\xR^d)$ and $s>0$ define the Sobolev norm:
        \[ \norm{u}_s^2 = \int \hat{u}(\xi)^2 (1+\abs{\xi}^2)^s d\xi.\]
        Let $H_s$ be the corresponding Sobolev space. Then, if $u\in H_n\cap H_{n+1}$ for some integer $n$, it is also in $H_s$ for $s\in[n,n+1]$, and: 
        \[ \norm{u}_s^2 \leq \norm{u}_n^{n+1 -s} \norm{u}_{n+1}^{s-n}.\]
    \end{lmm}
    \begin{proof}
        Decompose the integrand: $\hat{u}(\xi)^2 (1+\abs{\xi}^2)^s = \hat{u}(\xi)^{2-\alpha}(1+\abs{\xi}^2)^{s- \beta} \times \hat{u}(\xi)^{\alpha}(1+\abs{\xi}^2)^\beta$, and apply Hölder's inequality with $p = 1/(s-n)$, $q=1/(n+1 - s)$, $\alpha = 2(s-n)$ and $\beta = (n+1)(s-n)$.
    \end{proof}
    Now, we would like to bound: 
    \[
    A_{s,s'} = \int \left|
    (1-\Delta_\xi)^{s/2}(1 - \Delta_x)^{s'/2} \left(g(x,\xi) \chi(x-k,\xi - l)
	\right)
      \right|^2
      dxd\xi,
    \]
    for an $s$ in $(d/2,\floor{d/2} +1)$. Suppose for example that $d$ is even. The same resaoning as in the lemma gives the interpolation: 
     \begin{equation}
	 \label{eq=interpolation}
    A_{s,s'} \leq A_{d/2,s'}^{d/2 + 1 - s} A_{d/2 + 1,s'}^{s-d/2}.
     \end{equation}
     Since $d/2, d/2 +1$ and $s$ are integers, the quantities on the right hand side can be controlled
     using only classical derivatives:
     \begin{align*}
       A_{d/2,s'} &\leq C \sum_{\abs{\alpha}\leq s', \abs{\beta}\leq d/}
       \int \abs{\dXAlphaDXiBeta ( g\chi)}^2 dx d\xi \\
     &\leq C \sum C_{\alpha,\beta} \int \tilde{\chi}_{\alpha,\beta}
     \abs{\dXAlphaDXiBeta g}^2 dx d\xi,
 \end{align*}
 where the $\tilde{\chi}_{\alpha,\beta}$ are (norms of) derivatives of $\chi$.
 Since $\supp(\tilde{\chi}_{\alpha,\beta}) \subset \supp{\chi}$, and $\chi$ may
 be chosen such that the $\tilde{\chi}_{\alpha,\beta}$ are bounded, we may
 apply the estimates of proposition \ref{prpstn=gSatisfiesCV}, and obtain: 
 \[
 A_{d/2,s'} \aleq C \eps^{-2d} \lambda^{-2}.
 \]
 In the same way, one can derive a bound on $A_{d/2 +1,s'}$. 
 Interpolating between these two bounds, thanks to  \eqref{eq=interpolation}, gives the result.
 \end{proof}
\subsubsection{Bounds on the inverse \texorpdfstring{$\left(\pdo(h_1 - z)\right)^{-1}$}{operator}}
We  prove here the second item of proposition \ref{prpstn=boundsOnOperators},
going from $\pdo(g) = \pdo( (h_1 -z)^{-1})$ to $\pdo (h_1 - z)^{-1}$, we use the symbolic calculus. 
Let us write down the expansion of $(h_1 - z)\circ g$ given by
theorem~\ref{thrm=symbolicCalculus} (in the appendix). Since the derivatives of order $3$ of $h_1$ with respect to $\xi$ all vanish, the remainder $r_3$ is identically zero, and:
\begin{align}
    (h_1 - z)\circ g 
    &= 1 + \sum_i \partial_\xi^i(h_1 - z) D_x^i (g) + \sum_{i,j} \partial_\xi^{ij} (h_1 -z) D_x^{ij}g.\\
    &= 1 + \sum R_i + \sum R_{ij}.
    \label{eq=inversionHOne}
\end{align}
The operators appearing in the r.h.s. can now be bounded in $L^2$, using the same arguments as before (\ie{} bounds on the derivatives of their symbol in a local $L^2$ space): 
\begin{prpstn}
	The remainders $R_i$, $R_{ij}$ are bounded, and for all $\eta>0$, 
	\[ \norm{Op(R_i) } \aleq \lambda^{\eta} ; \qquad \norm{Op(R_{ij} } \aleq \lambda^{\eta}.\]
\end{prpstn}
\begin{proof}
    We follow the same scheme of proof as for proposition
    \ref{prpstn=gSatisfiesCV};
    however, each term $Q$ is multiplied by an additional $\eps^2 \xi_i
    \frac{r_0}{r}$. This modifies the bounds by a factor of $\lambda^{1/2 +
    1/\gamma}$, where $\gamma$ is the decay rate of $V$ (given by hypothesis 
    \ref{hpthss=decay}) (indeed, the critical region is the one where
    $V\approx \lambda$ and $\eps^2\xi^2 \approx \lambda$, so that
    $\eps^2 \xi_i \frac{r_0}{r} \approx \lambda^{1/2 + 1/\gamma}$).

    Therefore, the final operator bound will be:
    \[ \norm{Op(R_i)} \aleq \lambda^{1/2 + 1/\gamma -1 - \eta}\]
    Since $\gamma <2$, and $\eta$ is arbitrarily small, this concludes the
    proof.
\end{proof}

Therefore, we have defined a $G = \pdo(g)$ and an $R$ such that:
\[ (H_1 - z)G = I + R,\]
with $\norm{R}\leq C\eps$. For small $\eps$, $I+R$ is invertible, and
\[ (H_1 -z)G(I+R)^{-1} = I.\]
This shows that $H_1 - z $ is invertible, and the ``resolvent'' $R_1(z) = (H_1 -z)^{-1}$ is bounded:
\[ \norm{H_1 - z} \leq \norm{G} \frac{1}{1 - \norm{R}}.\]
This concludes the proof of proposition \ref{prpstn=boundsOnOperators}, page \pageref{prpstn=boundsOnOperators}.

\subsection{The Dirichlet part}
\label{sec=theDirichletPart}
We now define and study the auxiliary operator $H_0$,  which deals
with the Dirichlet boundary condition (\cf{} the explanation of the
general strategy in section \ref{sec=strategy}). Its definition is way simpler than that of $H_1$, we just put
\[
H_0 = \extRotOp + \epsilon\chiTldZero
\] 
where $\tilde{\chi}_0$ is $1$ at infinity, and is  supported outside $\supp{\chi_0}$ (\cf{} figure \ref{fig=indicators}). 

Once more, we would like to bound a resolvent associated to $H_0$.
\begin{prpstn}
  \label{prpstn=boundOnTheDirichletResolvent}
  There exists a $C$ such that, if $\abs{z-\lambda_j}\leq c_z\lambda_j$ (where $c_z$ is defined in proposition \ref{prpstn=boundsOnTheSymbol}), $H_0 -z$ has a bounded inverse, and
  \begin{equation}
    \| (H_0 - z)^{-1}\| \leq \frac{C}{\eps^2}.
  \end{equation}
\end{prpstn}
The main argument is positivity, and we will see, on the operator level, arguments that are reminiscent of the positivity bounds on the symbol in section \ref{sec=boundsOnTheSymbol}.

Since $z$ is exponentially small, it suffices to bound $H_0^{-1}$. We rotate it and study $e^{2\theta} H_0$. Recalling the expression \eqref{eq=expression_of_the_exterior_scaling} of $\extRotOp$, we get:
\begin{equation}
  e^{2\theta}H_0 =  -\eps^2 D^2 + \eps^2 \frac{e^{2\theta}\Lambda}{r_\theta^2}
  + e^{2\theta} (V_\eps(r_\theta,\omega) + \eps\chiTldZero).
\end{equation}
We now localize the so-called numerical range of $e^{2\theta}H_0$, \ie{} the set  $\{(e^{2\theta}H_0 \phi, \phi), \norm{\phi} = 1\}$.
\begin{lmm}
For any $\phi \in L^2$ with unit norm, 
\begin{itemize}
  \item $(-\eps^2 D^2 \phi,\phi) \in \xR^+$,
  \item $(\frac{e^{2\theta} \Lambda}{r_\theta^2} \phi, \phi)$ is in the cone $\{ \abs{\arg(z) }< \pi/4\}$,
  \item $(e^{2\theta}(V_\eps(r,\theta,\omega) + \eps \chiTldZero)\phi, \phi)$ is in the cone $\{ \abs{\arg (z- i(\eps) } < \pi/4 \}$, where $i(\eps)$ is given by \[ i(\eps) = \frac{1}{2} \inf_x (V_\eps(x) + \eps \chiTldZero(x)).\]
\end{itemize}
\end{lmm}
The first claim follows from the positivity of $-D^2$. To show the second one, it suffices to see that $e^{2\theta}r_\theta^{-2}$ is in the cone, to use the positivity of $\Lambda$ and then integrate over $r$. The proof of the third claim is similar to the positivity bounds in section \ref{sec=boundsOnTheSymbol} (and uses $\norm{\phi} = 1$).

This shows that the numerical range is included in the cone $\{ \abs{\arg(z -
i(\eps)) } < \pi/4 \}$, which is bounded away from $0$ in $\xC$ by at least
$i(\eps)$. Therefore, by a well known result of functional analysis, $H_0$ is
bounded by $i(\eps)^{-1}$. The choice of the cutoff function $\chiTldZero$
guarantees that $V + \eps\chiTldZero$ is greater than $\eps^2$ (\cf{} figure
\ref{fig=indicators}), so the bound of proposition 
  \ref{prpstn=boundOnTheDirichletResolvent} follows.


\subsection{Proof of the main bound}
\label{sec=putTogether}
We are finally in a position to prove theorem \ref{thrm=boundOnTheExteriorResolvent}. We do this by
reconstructing an approximate exterior resolvent from the Dirichlet part
$R_0$  and the ``infinity''  part $R_1$ (both depend on $z$).
Let $\chi$ be such that $\chi=1$ on $\supp \chi_1$, and 
$\tilde{R}(z) = R_0(z) \chi_0 + \chi R_1(z) \chi_1.$ 
Then $\tilde{R}(z)$ is our approximate resolvent.

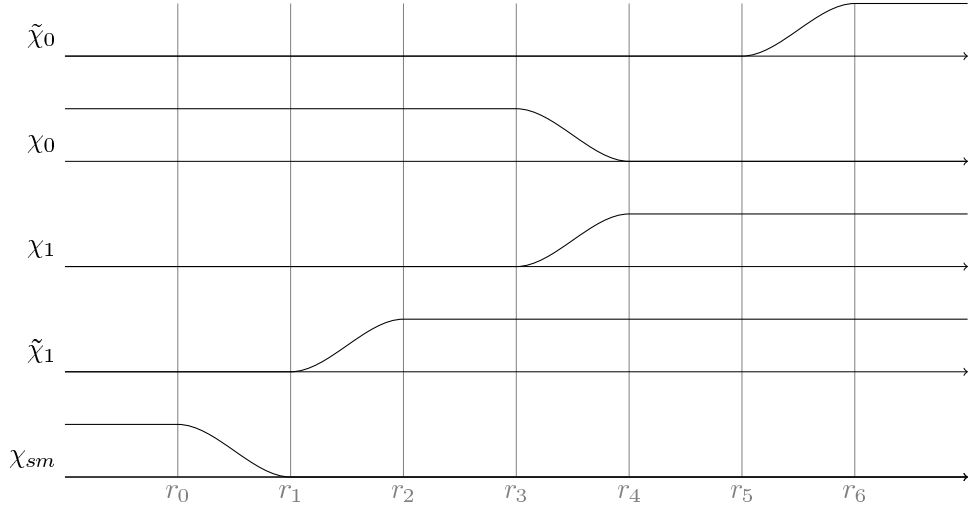
\begin{figure}
  \beginpgfgraphicnamed{indicators}%
\newcommand{\drawIndicator}[3]{%
    \draw[very thin,->] #2 +(-1,0) -- +(7,0);
    \draw #2 +(-1,0)
    node [above left] {#1} 
    --++(#3,0)..controls +(right:0.5cm) and +(left:0.5cm) .. +(1,1)
    --+(7 - #3,1);
    }
\newcommand{\drawReverseIndicator}[3]{%
    \draw[very thin,->] #2 +(-1,0) -- +(7,0);
  \draw #2  +(-1,0) node [above left] {#1}%
     ++ (0,1) +(-1,0) --++(#3,0)..controls +(right:0.5cm) and +(left:0.5cm) .. +(1,-1)--+(7 - #3,-1);
    }
\begin{tikzpicture}
    \begin{scope}[x=1.5cm,y=0.7cm]
        \draw[very thin,->] (-1,0) -- + (8,0);
    \foreach \x in {0,1,2,3,4,5,6} {%
      \draw[help lines] (\x,0) node [below] {$r_{\x}$} -- (\x,9);
    }
    \drawIndicator{$\tilde{\chi}_1$}{(0,2)}{1}
    \drawReverseIndicator{$\chi_{sm}$}{(0,0)}{0}
    \drawIndicator{$\chi_1$}{(0,4)}{3}
    \drawReverseIndicator{$\chi_0$}{(0,6)}{3}
    \drawIndicator{$\tilde{\chi}_0$}{(0,8)}{5}
\end{scope}
\end{tikzpicture}
\endpgfgraphicnamed
\label{fig=indicators}
\caption{The various indicator functions}
We use many cutoff functions to define our approximate resolvent. $\chiSmooth,
\chi_1$ and $\tilde{\chi}_1$ are used to deal with the ``infinity'' part
$R_1 = (H_1 - z)^{-1}$. We require the following: 
\begin{itemize}
    \item $H_1 = H$ on $\supp{\tilde{\chi_1}}$ (this is true since
	$\chiSmooth$ and $\tilde{\chi_1}$ have disjoint supports); 
    \item $\agmon(\supp(\nabla \tilde{\chi}_1), \supp(\chi_1))$ ``goes to
	infinity'' (\cf{} remark \ref{rmk=choiceOfPhi} for the precise
	hypothesis). 
\end{itemize}

The other indicators $\chi_0$ and $\tilde{\chi}_0$ deal with the Dirichlet
part $H_0$. They must satisfy: 
\begin{itemize}
    \item $V + \tilde{\chi}_0\eps$ is larger than $\eps^2$; 
    \item $\agmon(\supp(\chi_0), \supp(\tilde{\chi_0}))$ ``goes to infinity''
	(once more, \cf{} remark \ref{rmk=choiceOfPhi}).
\end{itemize}

All these conditions are met if we choose $r_j(\eps) = c_V \eps^{ - (1 +
(j/6))/\gamma}$ (this choice is of course largely arbitrary).
\end{figure}

\begin{prpstn}
    \label{prpstn=resolventsGluing}
  The operator $\tilde{R}(z)$ is bounded: 
  \begin{equation}
    \label{eq=RTildeIsBounded}
    \norm{\tilde{R}(z)} \leq \frac{C_\eta}{\lambda^{1+\eta}}.
  \end{equation}
  It is an approximate resolvent:
  \begin{equation}
  \label{eq=RTildeIsGood}
   (\extRotOp - z)\tilde{R}(z) = Id + \tilde{r}(\eps),
   \end{equation}
  where the remainder is such that $\norm{\tilde{r}(\eps)} = o(\eps)$.
\end{prpstn}
Once this is proved, the bound on the true resolvent follows. Indeed, for $\eps$ sufficiently small, the r.h.s. of
\eqref{eq=RTildeIsGood} is invertible, and its inverse is bounded by (say) $2$.
Multiplying \eqref{eq=RTildeIsGood} by this inverse, we get
\[
  (\extRotOp - z) \tilde{R}(z)(Id + \tilde{r}(\eps))^{-1} = Id.
\]
This implies that $(\extRotOp - z)$ is invertible, and, thanks to
\eqref{eq=RTildeIsBounded}, its inverse is bounded by
$2\frac{C_\eta}{\lambda^{1+\eta}}$. This will (finally!) end the proof of theorem
\ref{thrm=boundOnTheExteriorResolvent}.

\subsection{Proof of proposition \ref{prpstn=resolventsGluing}}
    Let us first decompose $(\extRotOp - z) \tilde{R} (z)$.
\begin{align}
    \notag
  (\extRotOp -z) \tilde{R} (z)
   &= \left(-\eps^2\Delta_\theta + V_\theta - z \right) \left(R_0(z) \chi_0 +
   \tilde{\chi}_1 R_1(z) \chi_1\right) \\
  \label{eq=reconstruction}
   &= (-\eps^2 \Delta_\theta + V + \eps\tilde{\chi}_0  - z)R_0(z) \chi_0 +
   (-\eps^2\Delta_\theta + V_1)\tilde{\chi}_1  R_1(z) \chi_1 \\
   \notag
   &\quad - \eps \tilde{\chi}_0 R_0(z)\chi_0 + (V-V_1)\tilde{\chi}_1  R_1(z) \chi_1 \\
\end{align}
The last term vanishes, because $(V-V_1)\tilde{\chi}_1 $ is zero. We commute
$\tilde{\chi}_1 $ and $(-\eps^2\Delta_\theta+V_1)$. Since $\chi_0 +\chi_1 = 1$, we get:
\begin{align}
    \notag
  (\extRotOp - z)\tilde{R}(z)
  &= Id + [-\eps^2\Delta_\theta,\tilde{\chi}_1 ]R_1\chi_1 - \eps\tilde{\chi}_0 R_0 \chi_0\\
  \label{eq=rZeroAndROne}
  &= Id + r_1 + r_0.
\end{align}
It remains to show that
the last two terms are small: $ \norm{r_0} + \norm{r_1} = o(\eps)$. The idea is
similar to the proof of the Agmon estimates (theorem \ref{thm=decayOfEigenfunctions}). It is made a bit more difficult by the fact that we are dealing with
the distorted operators and not with the usual Laplacian. 

Let us prove in some detail the bound on $r_0$. Let $v$ be in $L^2$, and let $u = R_0(z) \chi_0 v$. We would like to show: 
\begin{equation}
    \label{eq=boundOnChiTldZeroU}
    \norm{\chiTldZero u} \leq o(\eps)\norm{v}.
\end{equation}
 Let us recall the basic result used in the proofs of Agmon estimates (eq. \eqref{eq=basicAgmon}): 
\begin{equation}
     \label{eq=basicAgmonAgain}
    \eps^2 \int \abs{\nabla ( \exp(\phi/\eps) u)}^2 dx + \int ( V - \abs{\nabla \phi }^2) \exp( 2\phi/\eps) u^2 dx \\
    = \int \exp(2\phi/\eps) (-\eps^2 \Delta u + Vu) u dx.
\end{equation}
This is originally written for $\Omega$ a bounded domain, but the arguments of \cite{HS84} (in the proof of lemma 2.7) show that it extends to our case, if $\phi$ is constant at infinity. 

We need to choose a good function $\phi$. We impose: 
\begin{itemize}
    \item $\phi$ is radial;
    \item $\phi $ is constant on   $\supp \chiTldZero$  and on $\supp \chi_0$;
    \item Calling $S$ and $I$ the values of $\phi$ on $\supp\chiTldZero$,
	$\supp\chi_0$,  $S - I$ should go to  $\infty$ when $\eps\to 0$ ($S$
	and $I$ depend on $\eps$ through the choice of the functions $\chi_0$,
	$\chiTldZero$);
    \item $\abs{\nabla\phi}^2 \leq \frac{1}{2}V_\eps$.
\end{itemize}
\begin{rmrk}
    \label{rmk=choiceOfPhi}
    $\phi$ should be thinked of as (a multiple of) the Agmon distance to the
    support of $\chiTldZero$ (and made constant on $\supp \chi_0$). This
    particular choice cannot be made in general, since we need $\phi$ to be
    radial to perform our computations with the distorted operator $H_0$. 
    
    The hypotheses on $V$ make it easy to find a $\phi$ that works: choose $\phi(r) = r^\beta$ beween the supports, with $1<\beta<2- \gamma$. Since $\beta>0$, $S-I$ will go to infinity, and $\beta<2-\gamma$ ensures that $\abs{\nabla\phi}^2 \leq V_\eps/2$.
  \end{rmrk}

Let us begin our calculation: the idea is to apply equation \eqref{eq=basicAgmonAgain}, and the fact that on $\supp \chiTldZero$, $\phi$ is much larger than it is on $\supp \chi_0$. 
\begin{align*}
    \norm{\chiTldZero u}^2
    &= \int \chiTldZero (x) \abs{u(x)}^2 dx \\
    &\leq \left(\inf_{\chiTldZero >0}\left(
        (V_\eps + \eps\chiTldZero - \abs{\nabla \phi}^2) e^{2\phi/\eps}
    \right)\right)^{-1}
    \int e^{2\phi/\eps} (V_\eps + \eps\chiTldZero - \abs{\nabla\phi}^2) \abs{u}^2 dx.
\end{align*}
Let us write the inf as $\tilde{I}$, and apply \eqref{eq=basicAgmonAgain}.
\begin{equation}
    \label{eq=controlChiTldZeroU}
    \norm{\chiTldZero u}^2
    \leq \tilde{I}^{-1}
    \int e^{2\phi/\eps} ( -\eps^2\Delta + V_\eps + \eps\chiTldZero)u \cdot u dx.
\end{equation}
We want to use the fact that $H_0 u = v$: we need to bound the r.h.s. in terms of $H_0 = -\eps^2 \Delta_\theta + V_\theta$. We denote by
$A$ the function  such that $A(r(x)) = e^{2\phi(x)/\eps}$ (this is made possible by the requirement that $\phi$ should be radial).
\begin{align}
    \notag
   \forall x,  V_\eps(x) &\leq 2 \xRe( V_\theta(x)), \text{therefore}\\
   \label{eq=controlOfV}
    \int A(r(x)) V_\eps(x) \abs{u(x)}^2 
    &\leq 2 \xRe\left( \int V_\theta(x) \abs{u(x)}^2 dx \right).
\end{align}
We treat the kinetic part with a kind of ``sectoriality'' argument, using the decomposition \eqref{eq=decomposition} that helped us define the distorted operator: 
\begin{align*}
    \forall x, \frac{1}{r^2(x)} 
    &\leq\xRe\left( \frac{2}{r_\theta(x)^2} \right), \text{ therefore }\\
    \int A(r) (-\Delta u)\cdot u dx 
    &= \iint (-D^2u)u drd\omega
      + \int\frac{A(r)}{r^2} \int\Lambda u\cdot u d\omega dr \\
    &\leq 2\iint (-D^2u)u drd\omega + 2\int A(r) 
      \xRe\left(\frac{1}{r_\theta^2}\right) \int \Lambda u \cdot u d\omega dr \\
    &\leq 2 \xRe\left(\int A(r) (-D^2 + \frac{\Lambda}{r_\theta^2})u \cdot u  dx\right) \\
    &\leq 2 \xRe \left( \int A(r) (- \Delta_\theta u)u dx\right).
\end{align*}
Inserting this inequality and \eqref{eq=controlOfV} in \eqref{eq=controlChiTldZeroU} yields:
\begin{equation*}
    \norm{\chiTldZero u}^2 \leq \frac{2}{\tilde{I}} \xRe\left( \int A(r(x))H_0 u \cdot u dx\right).
\end{equation*}
Now, $H_0 u = \chi_0 v$, by definition  of $u$. By definition of $A$ (and of
$\phi$), $A(r) = \tilde{S} =  \exp(-2S/\eps)$ on $\supp\chi_0$. Therefore
\begin{equation}
    \norm{\chiTldZero u}^2 
    \leq \frac{2\tilde{S}}{\tilde{I}}  
    \times \xRe\left( \int \chi_0 v \cdot u dx\right).
\end{equation}
Use Cauchy--Schwarz on the right side: 
\[
    \norm{\chiTldZero u}^2 
    \leq \frac{2\tilde{S}}{\tilde{I}}  
    \norm{\chi_0v} \norm{u}\leq \frac{2\tilde{S}}{\tilde{I}} \norm{v}\norm{u}.
\]
Now we use our a priori bound on the resolvent $R_0$ (
proposition \ref{prpstn=boundOnTheDirichletResolvent}): $\norm{u} = \norm{R_0 \chi_0 v}\leq (C/\eps) \norm{v}$.
Since $\tilde{S} = \exp( -2S/\eps)$ and $\tilde{I} \geq \eps \exp(2I\eps)$, we
get:
\[
    \norm{\chiTldZero u}^2 
    \leq \frac{2C}{\eps^2}  
    \exp\left( \frac{2(I-S)}{\eps}\right) \norm{v}^2.
\]
Since $S-I$ goes to infinity, $\norm{\chiTldZero R_0 \chi_0 v} =
\norm{\chiTldZero u} = o(\eps) \norm{v}$, \eqref{eq=boundOnChiTldZeroU} holds and 
$r_0$ is indeed a small term.

\bigskip

Recall that the other term, $r_1$, is defined by
  \[
  r_1 =  [-\eps^2\Delta_\theta,\tilde{\chi}_1 ]R_1\chi_1 
  \]
  The commutator is explicit:
\[ 
  [-\eps^2\Delta_\theta, \tilde{\chi}_1 ] 
  = -\eps^2\left( 
    \Delta_\theta\tilde{\chi}_1  -   \nabla\tilde{\chi}_1  \nabla
    \right).
\]
Since $\supp{\nabla\tilde{\chi}_1 }$ is far (in Agmon distance) from
$\supp{\chi_1}$, an argument similar to the one we just developed for $r_0$
proves the same kind of estimate. Therefore proposition
\ref{prpstn=resolventsGluing} holds.


\section{Spectral stability} 
\label{sec=spectralStability}
In this section, we prove the stability of spectral quantities when we put a
boundary condition on the sphere~: near the eigenvalues of $\intOp$,
there must be eigenvalues of $\rotOp$, \ie{} resonances. We first prove an estimate
on the Dirichlet perturbation. Then, we get the existence of resonances and a first rough
localization result. Finally, we refine this localization and prove theorem \ref{thm=mainResult}.

\subsection{An estimate on the Dirichlet perturbation}
For any $a$ and $\theta$, we define, following \cite{CDKS87}, 
\[
W(\theta,a)  = R(\theta,a) - R^D(\theta,a).
\]
This describes how much the Dirichlet boundary condition changes
the solution of the equation $\fullOp u = \phi$. 

The idea is to express the perturbation in terms of trace
operators on the boundary of the sphere, and then use Agmon estimates (inside
and outside the sphere) to control the traces.

Let $T_e$, $T_i$ be the trace operators on the outside and inside of the
sphere. Still following \cite{CDKS87}, define
\begin{equation}
  \label{eq=AsAndBees}
\begin{aligned}
  A  (\theta,a) &= T_i (\rotOp - a)^{-1},\\
  B_e(\theta,a) &= e^{-3\theta/2} T_e \nabla_r (\extRotOp - a)^{-1}, \\
  B_i(\theta,a) &= T_i \nabla_r (\intOp - a)^{-1}, \\
  B(\theta,a)   &= B_i \oplus B_e.
\end{aligned}
\end{equation}
where $\nabla_r$ is the radial derivative. We also define $\PDrchlt$ as the spectral projector on the span of the exponentially small eigenvalues of the Dirichlet operator, and $\QDrchlt = 1 - \PDrchlt$.
\begin{prpstn}[\cite{CDKS87}, eq. (3.2)]
    The perturbation $W$ can be decomposed in the following way:
    \begin{equation}
        \label{eq=decompositionOfW}
        W(\theta,a) = \eps^8B^*T(H - a)^{-1}T*B.
  \end{equation}
\end{prpstn}
We show the following:
\begin{thrm}
    \label{thrm=boundOnW}
    For any $\delta$, there is a $C_\delta$ such that:
    \begin{align}
        \label{eq=boundOnA}
      \norm{A} &\leq C_\delta e^{\delta/\eps}, \\
        \label{eq=boundOnB}
      \norm{B} &\leq C_\delta e^{\delta/\eps}, \\
      \label{eq=boundOnW}
      \norm{W} &\leq C_\delta e^{\delta/\eps}.
    \end{align}
    Moreover, $B$ is very small on the Dirichlet eigenspaces~:
    \begin{equation}
        \label{eq=BPDisSmall}
        \norm{B \PDrchlt}
        \leq C_\delta \exp \left (- \frac{S(\eps) - \delta}{\eps} \right)
    \end{equation}
    \end{thrm}
The remainder of the section is the proof of this result. 
We do it in several steps.
\step{Bounds on $B_i$ and $B_e$} 
    We detail the bounds on the interior part $B_i$. We use the following to control the trace operator:
    \begin{thrm}
        \label{thrm=traceControl}
        For all $u\in H^1$, and all cutoff function $\chi$ supported near the boundary of the ball,
        \[ \norm{T_i u}^2 \leq 2\norm{\chi u}\norm{ \nabla\chi u}\]
    \end{thrm}
    For a proof, see \cite{CDKS87} (lemma 4). 
   
    So, it is enough to show:
    \begin{align}
        \label{eq=traceControl_1}
        \norm{\chi (\intOp -a)^{-1} u} 
	&\aleq \eps^{-3/2} \norm{u}, \\
        \label{eq=traceControl_2}
        \norm{ \nabla_r\chi (\intOp -a)^{-1} u}
        &\aleq C_\delta \exp(\delta/\eps) \norm{u}, \\
        \label{eq=traceControl_3}
        \norm{ \nabla_r\chi\nabla_r (\intOp -a)^{-1} u}
        &\aleq C_\delta \exp(\delta/\eps)\norm{u}.
    \end{align}
    The control \eqref{eq=traceControl_1} follows from Agmon estimates. Indeed, we know that the spectrum of the interior operator consists of two parts. Let $S$ be the Agmon distance between the minima and the ball of radius $\rZero$. Then 
    \begin{align*}
        \norm{\chi (\intOp - a)^{-1} u }
        &\leq \norm{\chi (\intOp - a)^{-1} \PDrchlt u}
          + \norm{ (\intOp - a)^{-1} \QDrchlt u}\\
        &\aleq C_\delta \exp\left( -\frac{S - \delta}{\eps}\right)\norm{u}
        + \eps^{-3/2}\norm{u}, 
    \end{align*}
    where the first bound is the Agmon estimate of theorem \ref{thm=decayOfEigenfunctions} on eigenfunctions coming from small eigenvalues, and the second bound comes from the fact that $(\intOp -a)^{-1}$ is bounded by $C/\eps$ on the range of $\QDrchlt$.

    We now turn to the proof of \eqref{eq=traceControl_2}. Once more, decompose $u$ as $Pu + (u-Pu)$. To bound $ \norm{\nabla_r\chi (\intOp -a)^{-1} Pu}$, we use theorem \ref{thm=decayOfEigenfunctions} again:
    \begin{align*}
      &\norm{\nabla_r\chi(\intOp - a)^{-1}   \PDrchlt u}\\
      &\quad\leq 
        \exp\left( -  \inf_{\supp\chi}d(x) / \eps \right)
      \norm {\exp(d(x)/\eps)\abs{\nabla (\intOp - a)^{-1} Pu} } \\
      &\quad\leq 
        C_\delta \exp\left(
          (\delta -  S(\eps) / \eps 
        \right)
    \end{align*}
    Since $\chi$ is localized on the boundary, the r.h.s. is $o(\eps)$, and \eqref{eq=traceControl_2} is proved (for $\PDrchlt u$).

    To bound the other term, we apply \eqref{eq=basicAgmon} once more, with $u = (\intOp - a)^{-1}\QDrchlt u$: 
    \begin{align*}
      &\norm{\nabla_r\chi(\intOp - a)^{-1} \QDrchlt u}^2 \\
      &\quad \leq \exp\left( -2  \inf_{\supp\chi}d(x) / \eps \right)
      \norm {\exp(d(x)/\eps)\abs{\nabla (\intOp - a)^{-1}\QDrchlt u}}^2 \\
      &\quad \leq \exp\left( -2  \inf_{\supp\chi}d(x) / \eps \right)
      \int \exp(2d(x)/\eps) (1-P)u \cdot (\intOp - a)^{-1}\QDrchlt u dx\\
      &\quad\leq  \exp\left( \left(\sup d(x) -  \inf_{\supp\chi} d(x)\right) / \eps \right)
      \norm{\QDrchlt u}\norm{(\intOp -a)^{-1} \QDrchlt u} \\
      &\quad\leq C_\delta \exp\left(\delta/\eps\right) \eps^{-3/2} \norm{u}^2,
    \end{align*}
    where we successively introduce the Agmon distance, use \eqref{eq=basicAgmon}, use Cauchy--Schwarz and finally use the easy bound on the resolvent restricted to the range of $\QDrchlt $. Therefore, \eqref{eq=traceControl_2} holds. 

    We turn to the proof of \eqref{eq=traceControl_3}: this may be reduced to the previous estimates. To see it, first remark that, now that \eqref{eq=traceControl_1} and \eqref{eq=traceControl_2} are known, it is enough to show a bound on $\norm{\nabla_r\nabla_r \chi (\intOp - a)^{-1}}$ (commute the $\chi$ through the $\nabla_r$, and use the previous bounds with $\nabla\chi$ instead of $\chi$). The double radial derivative is now relatively bounded with respect to the Laplacian (cf. for example \cite{Kle86}): 
    \[
    \exists C, \forall v,
    \quad \norm{\nabla_r \nabla_r v} \leq C(\norm{v} + \norm{\Delta v}).
    \]
    The Laplacian is itself relatively bounded w.r.t. $\intOp$ (because the potential $V_\eps$ is bounded): 
    \[
    \norm{\nabla_r\nabla_r v}
    \leq C\norm{v} + \frac{C}{\eps^2} \norm{(\intOp -a) v}
    \]
    Setting $v = \chi(\intOp -a)^{-1}u$, and using \eqref{eq=traceControl_1} and \eqref{eq=traceControl_2} again, we obtain \eqref{eq=traceControl_3}.

    Thus, we have bounded $B_i$. The proof also shows that, on the range of $\PDrchlt$, the much better bound \eqref{eq=BPDisSmall} holds. Using similar arguments, we can control the exterior operator $B_e$ and get \eqref{eq=boundOnB}.


\step{Bound on $T R(a) T^*$}
The proof of this bound is a straightforward adaptation of the proof of the
simiular result in \cite{CDKS87} (equation 3.7 of that reference), and uses the
same arguments we just applied to bound $B$. Indeed, since $T^*$ maps
continuously the sphere into $L^2$, it suffices to bound $TR(a)u$. Using
theorem \ref{thrm=traceControl}, this reduces to bounds on $\chi R(a)u$ and
$\nabla_r\chi R(a)u$, for $\chi$ supported near the boundary. These are in turn
obtained by Agmon-type estimates. 

\step{Conclusion}
The bounds on the different parts of $B$ (first and second steps) prove that \eqref{eq=boundOnB} holds; the bound on $TR(a)T^*$, together with the decomposition \eqref{eq=decompositionOfW} of $W$ and the bound on $B$, show \eqref{eq=boundOnW}. This concludes the proof of theorem \ref{thrm=boundOnW}.
%
\subsection{Existence of resonances}
We follow the proof of \cite{CDKS87} (lemma 3 and theorem 4) and use the
classical argument of integrating resolvents on a contour to get spectral
projections on appropriate eigenspaces. 
More precisely, for each $j$ (and each eigenvalue $\lambda_j$ of the interior
operator), we define the contour $\Gamma_j(\eps)$ to be the circle of radius
$\lambda_j/2$, centered in $\lambda_j$. We would like to show that $
\int_{\Gamma_j} R(\theta,z) dz$ and $\int_{\Gamma_j} R^D(\theta,z)dz$ have the
same rank. We change variables: let $a(\eps) = - \exp(-d_a/\eps)$, where $d_a > d_1$, 
and let $\tilde{z} = \frac{1}{z-a}$. The first integral above exists if and
only if
\[
  \int_{\tilde{\Gamma}_j} ( R(\theta,a ) - \tilde{z} )^{-1} d\tilde{z} 
\]
exists, and if so, their values are equal.

We have all the ingredients to estimate the integrand.
\begin{prpstn}
  \label{prpstn=RDMinusZTildaSmall}
  The quantity $(R^D(\theta,a) - \tilde{z})^{-1}$ is $O(\lambda_j^{1-\eta}(\eps))$ when $\eps\to0$, uniformly on $\tilde{z} \in \tilde{\Gamma_j}$.
\end{prpstn}
\begin{proof}
  We follow closely the proof of lemma $3$ in \cite{CDKS87}. We rewrite our
  quantity as 
  \[
    (R^D(\theta,a) - \tilde{z})^{-1}  = - (z-a) - (z-a)^2 R^D(\theta, z).
  \]
  Since $a = a(\eps) = -\exp(-d_a/\eps)$ is much smaller than $\lambda_j$,
  $\abs{z-a} \aleq \abs{\lambda_j}$. The Dirichlet resolvent $R^D$ is estimated
  separately on the interior and on the exterior. On the interior part, we
  estimate it by the inverse distance to the spectrum: $z$ is on $\Gamma_j$,
  therefore its distance to $\sigma(\intOp)$ is at least of the order of
  $\lambda_j$. On the exterior part, we use the bound
  \ref{thrm=boundOnTheExteriorResolvent}. We obtain:
  \begin{align*}
    \norm{(R^D(\theta,a) - \tilde{z})^{-1} }
    &= O(\lambda_j) + \lambda_j^2 \left( \frac{C}{\lambda_j} 
       + \frac{C}{\lambda_j^{1+\eta}} \right) \\
     &= O(\lambda_j^{1-\eta}).
  \end{align*}
\end{proof}
 
Recall that $W(\theta,a) = R(\theta,a) - R^D(\theta,a)$, and let us define, for $\zeta\in\xC$, $\abs{\zeta} \leq 2$,
\[
A_\zeta (\tilde{z})
  = \int_{\Gamma_j}
    (R^D(\theta,a) - \tilde{z})^{-1}
    \sum_n \left(
      W \left( R^D(\theta,a) - \tilde{z} \right)^{-1} \zeta 
    \right)^n.
\]
Now, $\norm{W\times (R^D(\theta,a) - \tilde{z})^{-1}}$ is $o(1)$, thanks to
theorem \ref{thrm=boundOnW} and proposition \ref{prpstn=RDMinusZTildaSmall}. Therefore,  the series on the r.h.s. converges when $\eps$ is small,
$A_\zeta$ is analytic in $\zeta$, and uniformly bounded in $\tilde{z}$.

For $\zeta = 0$, we recover $(R^D(\theta,a) - \tilde{z})^{-1}$. For $\zeta =
1$, it is easy to see that $A_1(R(\theta,a) - \tilde{z}) = Id$, therefore
$(R(\theta,a) - \tilde{z})$ is invertible with inverse $A_1(\tilde{z})$.

We now consider the contour integral:
\begin{equation}
    \label{eq=contourIntegral}
    P_\zeta = -(2i\pi)^{-1} \int_{\tilde{\Gamma}_j} A_\zeta(\tilde{z})  d\tilde{z}
\end{equation}
The boundedness guarantees the existence of the integral. Moreover, $P_\zeta$ depends analytically on $\zeta$. Since $P_0$ is the spectral projection on the eigenstate corresponding to $\lambda_j$, the analyticity shows that $P_1$ is still a projection on a one dimensional space, and there exists a unique eigenvalue of $H(\theta)$ inside the contour $\Gamma_j$. This eigenvalue is the resonance of the original operator that we were looking for.
%
\subsection{Refined estimates on resonances} 
The proof of our main result (theorem \ref{thm=mainResult}) is not yet complete. Up to this point, 
we only know that there exists a resonance (say $\mu_i(\eps)$) inside a contour $\Gamma_j$ around the eigenvalue $\lambda_i(\eps)$, where the radius of the circle $\Gamma_j$ is of the order of $\lambda_j$. In this section, we indicate how to refine the estimation to get the stronger estimates annouced in theorem \ref{thm=mainResult}.

We follow the strategy of \cite{CDKS87}, section V, to prove the following result~:
\begin{prpstn}
    \label{prpstn=tunnellingExpansion}
  There exists functions $t_i(\eps)$, $\sigma_{n,i}(\eps)$ such that $\mu_i(\eps)$ has the following development~:
  \[ \mu_i(\eps) = \lambda_i(\eps) + \sum_n \sigma_{n,i}(\eps) \frac{t(\eps)^n}{n!},\]
  where the $\sigma_{n,i}$ are uniformly bounded in $n,\eps$, and
  \[ t_i(\eps) = O \left(\exp \left(-\frac{S(\eps)} {\eps} \right)\right).\]
\end{prpstn}
In particular, this result entails the estimates announceed in theorem \ref{thm=mainResult}.

The idea is the following~:
\begin{itemize}
  \item First, define $\tldLambda,\tldMu$ by the same change of variables as before~:
    \[
    \tldLambda = \frac{1}{\lambda_i(\eps) - a(\eps)},\quad
       \tldMu =     \frac{1}{\mu_i(\eps)     - a(\eps)}.
    \]
  \item Use the description of the Dirichlet perturbation by trace operators to get an implicit equation on $\tldMu$~:
    \[ \tldMu - \tldLambda = t(\theta) \sigma( \theta, \tldMu),\]
    where $t(\theta)$ and $\sigma( \theta, z)$ are given explicitly in terms of traces (\cf{} infra).
  \item Show that $t$ and $\sigma$ are independant of $\theta$, and estimate them, as well as 
    \[ 
    \tilde{\sigma}_{n,i} 
    =  \left( \left(\frac{d}{dz} \right)^{n-1} (\sigma^n (\tilde{z}) \right) |_{\tilde{z} = \tldLambda}
    \] 
  \item Use a refined implicit function theorem (Lagrange's inversion formula) to obtain a power series expansion of $\tldMu$:
    \[ \tldMu =  \tldLambda + \sum \tilde{\sigma}_{n,i} t^n / n!.\]
  \item Go back to $\mu_i(\eps)$, $\lambda_i(\eps)$.
\end{itemize}
These steps are detailed in \cite{CDKS87}. The $t$ and $\sigma$ are given by (\cite{CDKS87}, proof of theorem V.2):
\begin{align*}
  t(\theta) &= 
    \eps^2 \tr \abs{ B(\theta,a) \PDrchlt A^*(\bar{\theta}, a) }, \\
  \sigma(\theta, \tilde{z}) &=
    \eps^2 t(\theta)^{-1} \tr \PDrchlt A^*(\bar{\theta}, a) (1 - M(\theta,\tilde{z})) B(\theta,a) \PDrchlt,\\
  M(\theta,\tilde{z}) &=
  \eps^2 B(\theta,a) \QDrchlt (\QDrchlt R(\theta,a) \QDrchlt - \tilde{z})^{-1} \QDrchlt A^*(\bar{\theta}, a).
\end{align*}
Analyticity arguments show that these quantities do not depend on $\theta$. 
Since these algebraic formul\ae{} still hold in our case, all we have to do is 
to get estimates on $t$, $\sigma$ and the $\tilde{\sigma}_{n,i}$ and $\sigma_{n,i}$. First, putting together the bound \eqref{eq=BPDisSmall} on $B\PDrchlt$  and the bound \eqref{eq=boundOnA} on $A$, we get: 
\[ t = \bigO\left( \exp \left( - \frac{S(\eps) - \delta}{\eps}\right)\right).\]
Next, $\sigma$ is controlled by:
\begin{align*}
    \sigma &\leq \eps^2 t^{-1} \norm{1-M} \tr\abs{B(\theta,a) \PDrchlt A^*} \\
    &\leq \norm{1 - M},
\end{align*}
so it suffices to get a bound on $M$. Following the arguments of proposition \ref{prpstn=RDMinusZTildaSmall}, one can show that $(\QDrchlt R^D(\theta,a) \QDrchlt - \tilde{z})^{-1}$ is exponentially small. The  same interpolation argument that was developed after \eqref{eq=contourIntegral} then shows that $(\QDrchlt R(\theta,a)\QDrchlt - \tilde{z})^{-1}$ is exponentially small. This compensates the $\exp(\delta/\eps)$ coming from the bounds on $A$ and $B$, thereby proving that $M$ is $\bigO(1)$. 

We can then go back to \cite{CDKS87} to see that this justifies the steps explained above, proving the power series expansion of proposition \ref{prpstn=tunnellingExpansion} and ending the proof of our main result.

 \section{Appendix: on symbols and pseudo-differential calculus}
 We found it convenient to use well known techniques of microlocal analysis and pseudo-differential operators (\PDO) to approximate the resolvent. We briefly describe these techniques in this appendix. For more detailed presentations and historical remarks, see \eg{} \cite{Tay81,Hor85,Ste93}.
 \subsection{Pseudo-differential operators}
 \label{sec=pseudors}
  It is well known that a (classical) differential operator with constant
  coefficients $A = \sum a_\alpha D^\alpha$ (where $D^\alpha =
  i^{-\abs{\alpha}} \partial_x^\alpha$) becomes, in the Fourier space, a
  multiplication by a polynomial in $\xi$, called the \emph{symbol} of the
  operator. More precisely, if $\fou{u}(\xi) = (2\pi)^{-d/2} \int u(x)
  e^{-ix\cdot \xi} dx$ denotes the Fourier transform, we have
  \[
  \fou{A\phi}(\xi) = a(\xi) \fou{\phi}(\xi),
  \]
  where $a(\xi) = \sum_\alpha a_\alpha \xi^\alpha$.
  In particular, the operator is inversible if the symbol is nowhere zero, and the inverse corresponds to multiplication by $a^{-1}$.

  Trying to generalize this to non-constant coefficients naturally leads to considering operators defined by a symbol $a(x,\xi)$ which is not necessarily a polynomial, and try to define $\op(a)$ by: 
\[
  \op(a) f:x \mapsto (2\pi)^{-d/2} \int e^{i(x,\xi)} a(x,\xi) \hat{f}(\xi) d\xi
\]
To make sense of this definition, hypotheses on $a$ are needed, and many classes of ``good'' $a$ have been defined. Such classes allow ``symbolic calculus'', \ie{} results that allow one to work on symbols rather than operators: typically, one would like to compare $\pdo(a)\pdo(b)$ by $\pdo(ab)$.

Among these classes, we mention the classical classes $\mathcal{S}^m$, defined by decay conditions on $\xi$:
\[
  a\in \mathcal{S}^m \iff \norm{\dXAlphaDXiBeta a}_\infty 
  \leq \frac { C_{\alpha,\beta}}{ \langle \xi \rangle^{m+\abs{\beta}}},
\]
where $\langle\xi\rangle = (1+\xi^2)^{1/2}$. This decay in $\xi$ implies that, for $m>0$,  operators in $\mathcal{S}^m$ improve differentiability by $m$ units. On these classes, the following holds:
\begin{thrm}
    [Symbolic calculus,\cite{BF74}, Theorem 1]
    \label{thrm=symbolicCalculus}%
    If $a\in\mathcal{S}^m$, $b\in\mathcal{S}^n$, then $\pdo(a)\pdo(b)$ is a \PDO, its symbol $a\circ b$ is in $\mathcal{S}^{m+n}$, and has the following expansion:
    \begin{equation}
        \forall N, \quad a\circ b = \sum_{\abs{\alpha}<N} \frac{1}{\alpha!}
          \partial_\xi^\alpha a (x,\xi) D_x^\alpha b (x,\xi)
           + r_N
        \label{eq=asymptoticExpansion}
    \end{equation}
    where $r_N\in \mathcal{S}^{m+n -N}$, and is defined by the integral:
    \begin{equation}
        r_N = \sum_{\abs{\alpha} = N} c_\alpha \int_0^1 \iint_{\xR^{2d}}e^{-i\langle y-x,\eta - \xi\rangle } \partial_\xi^\alpha a(x,\eta_t)  D_y^\alpha b(y,\xi) dyd\eta dt
        \label{eq=remainder}
    \end{equation}
    for some constants $c_\alpha$ and $\eta_t = \xi + t(\eta - \xi)$.
\end{thrm}
\begin{rmrk}
    We cite \cite{BF74}, where the result is more general, because the remainder there is explicit. This theorem can be found in any of the textbooks mentioned above.
\end{rmrk}

The derivative estimates show that our symbols always belong to some $\mathcal{S}^m$.

\subsection{\texorpdfstring{$L^2$}{L2} bounds}
The \PDO{} are defined at first on the Schwartz space; proving that they send $L^2$ to $L^2$ is a well studied problem, and many criteria are available.

The first results are written for the classical symbols (in $\mathcal{S}^m$). If $m\leq 0$, they define bounded operators in $L^2$. In our case, the boundedness is therefore easy to prove. However, the estimates of the $L^2$ norm depend on the $L^\infty$ norm of many derivatives of the symbol, and are insufficient to carry on the stability argument of section \ref{sec=spectralStability}.

Finding precisely how many derivatives are needed for $L^2$ continuity to hold has been the subject of much work (following Calderon and Vaillancourt's \cite{CV72}, see \eg{} \cite{Hwa87,CM78}). These refined estimates do not yet give the desired result. They are still stated in terms of uniform norms of the derivatives, and our symbols behave badly in this respect.

To understand the difficulty, consider the symbol $r(x,\xi) = 1/(\lambda +\xi^2)$. It is a gross simplification of our symbol, but it retains its main features. The fact that $Op(r)$ should be an approximate resolvent for the Laplacian, and the trivial positivity bound of the latter lead us to believe that $Op(r)$ should be bounded in $L^2$ by $c/\lambda$~: if this bound can be obtained by pseudo differential arguments (without resorting to positivity, nor on the fact that it does not depend on $x$), it should carry over to our symbol. 

However, it is easily seen that the best uniform bound on $\partial_\xi^{\beta}r$ when $\beta_i = 0$ or $1$,  is of the order $\lambda^{-1-\abs{\beta}/2}$, and even the restricted conditions of Calderon and Vaillancourt ($\abs{\beta}\leq \floor{n/2} + 1$, or $\beta_i\in\{0,1\}$) do not give a bound of the right order.

Fortunately, subsequent papers  have shown  still other conditions for $L^2$ continuity. In particular, A.~Boulkhemair (in \cite{Bou95}, to which we refer for further reference) , expliciting results from \cite{BM88}, gives a statement which involves local $L^2$ norms of the symbol. 

\begin{thrm}
  \label{thm=CalderonVaillancourt}
  [\cite{Bou95}, Corollary 3]
Let $\chi$ be a bump function in $\xR^{2d}$ ($\chi$ is compactly supported and normalized in $L^2$), and $s>d/2$, $s'>d/2$.
  Let $a:\xR^d \times \xR^d$  be such that
  \begin{equation}
    \label{eq=conditionCalderon}
    \exists C(a), \qquad
    \int \abs{ (1-\Delta_x)^{s/2} (1-\Delta_\xi)^{s'/2}\left( \chi(x-k, \xi-l) a(x,\xi)\right) }^2 dx d\xi \leq C(a)^2,
  \end{equation}
for all $(k,l)$ in $\xR^d \times \xR^d$.

  Then $a(x, D)$ is continuous from $L^2$ to $L^2$  with its norm bounded by
  $C_{s,d}C(a)$ , where $C_{s,d}$ only depends on $s,d$ and $\chi$.
\end{thrm}
This condition, on the toy symbol $r$, gives a bound of the order $\lambda^{-1-(s-d/2)}$: we almost recover the right order $\lambda^{-1}$. 


\subsection{Some symbols}
Finally, we record here the symbols of various operators. The radial part of the Laplacian, $D^2$, has the following symbol:
\[
  \sigma(D^2) (x,\xi) = \frac{(n-1)(n-3)}{4r(x)^2} + i \frac{n-1}{r^2(x)} x\cdot \xi - \frac{1}{r^2(x)} (x\cdot\xi)^2.
\]
It is then easily seen that the symbol of the scaled Laplacien is given by:
\begin{align}
  \notag
\sigma( - \Delta_\theta ) 
&= \left( \frac{r^2}{r_\theta^2} - e^{-2\theta} \right) \sigma( D^2) + \frac{r^2}{r_\theta^2} \abs{\xi}^2 \\
\label{eq=symbolDilatedLaplacian}
&= e^{-2\theta} \abs{\xi}^2 + \left( \frac{r^2}{r_\theta^2} - e^{-2\theta} \right)\left(\abs{\xi}^2 + \sigma(D^2)\right)
\end{align}
for $r(x)>r_0$, and $\sigma(-\Delta_\theta) = \abs{\xi}^2$ for $r<r_0$. 



\end{document}